\documentclass[12pt]{amsart}
\usepackage{amsmath}	
\usepackage{amssymb}
\usepackage{caption}
\usepackage{tikz}
\usetikzlibrary{calc, decorations.markings}

\newtheorem{theoreme}{Th\'eor\`eme}[section]
\newtheorem{cor}[theoreme]{Corollaire}
\newtheorem{lemme}[theoreme]{Lemme}
\newtheorem{proposition}[theoreme]{Proposition}

\numberwithin{equation}{section}

\newcommand{\bA}{\mathbb{A}}
\newcommand{\bC}{\mathbb{C}}
\newcommand{\bN}{\mathbb{N}}

\newcommand{\bQ}{\mathbb{Q}}
\newcommand{\bR}{\mathbb{R}}
\newcommand{\bZ}{\mathbb{Z}}

\newcommand{\cC}{{\mathcal{C}}}
\newcommand{\cD}{{\mathcal{D}}}
\newcommand{\cE}{{\mathcal{E}}}
\newcommand{\cG}{{\mathcal{G}}}
\newcommand{\cH}{{H}} 
\newcommand{\cK}{{\mathcal{K}}}

\newcommand{\cM}{{\mathcal{M}}}
\newcommand{\cO}{{\mathcal{O}}}

\newcommand{\cR}{{\mathcal{R}}}

\newcommand{\cU}{{\mathcal{U}}}

\newcommand{\card}{\mathrm{Card}}

\newcommand{\dz}{\mathrm{d}\/z}

\newcommand{\et}{\quad\text{et}\quad}

\newcommand{\Mat}{\mathrm{Mat}}

\newcommand{\Qbar}{\bar{\bQ}}
\newcommand{\ssi}{\quad\Longleftrightarrow\quad}

\newcommand{\tcC}{\tilde{\cC}}
\newcommand{\tgamma}{\tilde{\gamma}}
\newcommand{\tG}{\tilde{G}}

\newcommand{\ua}{\mathbf{a}}

\newcommand{\ue}{\mathbf{e}}

\newcommand{\un}{\mathbf{n}}

\newcommand{\uun}{\mathbf{1}}

\newcommand{\disp}{\displaystyle}

\renewcommand{\cG}{G}
\renewcommand\Re{\operatorname{Re}}
\renewcommand\Im{\operatorname{Im}}

\tikzset{->-/.style={decoration={
  markings,
  mark=at position #1 with {\arrow{>}}},postaction={decorate}},
  ->-/.default=0.5,
  }

\begin{document}

\baselineskip=16pt

\title[Approximation des valeurs de la fonction exponentielle]
{Approximation simultan\'ee des valeurs de la fonction exponentielle dans les ad\`eles}
\author{Damien Roy}

\subjclass[2010]{Primaire 11J13; Secondaire 11J61, 11J82, 11H06.}
\keywords{ad\`eles, approximations d'Hermite, calcul de volumes,
chemins de descente maximale, fonction exponentielle, fractions continues,
g\'eom\'etrie des nombres, mesures d'approximation, racines des polyn\^omes,
semi-r\'esultant.}
\thanks{Recherche support\'ee en partie par une subvention \`a la d\'ecouverte du CRSNG}

\begin{abstract}
On montre que les approximations d'Hermite des valeurs de la
fonction exponentielle en des nombres alg\'ebriques distincts
sont essentiellement optimales quand on les
consid\`ere d'un point de vue ad\'elique, c'est-\`a-dire quand on
prend en compte les quotients de ces valeurs qui ont un sens dans les
diff\'erents compl\'et\'es (archim\'ediens ou $p$-adiques) d'un corps
de nombres contenant ces nombres alg\'ebriques.
\end{abstract}

\maketitle

%
%

\section{Introduction}
\label{sec:intro}

On sait, gr\^ace \`a Euler, que le nombre $e$ poss\`ede un d\'eveloppement
en fraction continue qui consiste de progressions arithm\'etiques imbriqu\'ees
\[
 e=[2,(1,2n,1)_{n=1}^\infty] = [2,1,2,1,1,4,1,1,6,1,\dots].
\]
Euler, Sundman et Hurwitz ont aussi montr\'e un r\'esultat semblable
pour les nombres $e^{2/m}$ o\`u $m$ est un entier non nul \cite[\S\S31-32]{Pe1929}.
En cons\'equence, on peut donner de tr\`es bonnes
mesures d'approximation rationnelle de ces nombres (voir par exemple les
r\'esultats enti\`erement explicites de Bundschuch \cite[Satz 2]{Bu1971},
pour le cas o\`u $m$ est pair).
C'est cette derni\`ere propri\'et\'e qui nous int\'eresse ici.
Nous en proposons l'explication heuristique suivante: 
les quotients $2/m$ avec $m\in\bZ\setminus\{0\}$ sont
les seuls nombres rationnels $z$ non nuls pour lesquels la s\'erie usuelle
\begin{equation}
 \label{intro:eq:e^z}
 e^z=\sum_{k=0}^\infty \frac{z^k}{k!}
\end{equation}
ne converge qu'au sens r\'eel.  En effet, soit $p$ un nombre premier et
soit $\bC_p$ le compl\'et\'e de la cloture alg\'ebrique $\Qbar$ de $\bQ$
pour la valeur absolue $p$-adique de $\bQ$ \'etendue \`a $\Qbar$, avec
$|p|_p=p^{-1}$.  On sait que, pour $z\in\bC_p$, la s\'erie \eqref{intro:eq:e^z}
converge dans $\bC_p$ si et seulement si $|z|_p<p^{-1/(p-1)}$. En particulier, pour
un nombre rationnel $z$, vu comme \'el\'ement de $\bC_p$, cette s\'erie
converge si et seulement si le num\'erateur de $z$ est divisible
par $p$ lorsque $p\neq 2$, et par $4$ lorsque $p=2$.

Ce ph\'enom\`ene se g\'en\'eralise aux nombres alg\'ebriques. En effet,
soit $K$ un corps de nombres, c'est-\`a-dire une
extension alg\'ebrique de degr\'e fini de $\bQ$. Alors toute valeur
absolue sur $K$ induit la m\^eme topologie sur $K$ que celle provenant d'un
plongement de $K$ dans $\bC$ ou dans $\bC_p$ pour un nombre premier $p$.
On dit que de tels plongements d\'efinissent la m\^eme place $v$ de $K$
s'ils induisent la m\^eme valeur absolue sur $K$ not\'ee $|\ |_v$.  On
d\'esigne alors par $K_v$ le compl\'et\'e de $K$ pour cette valeur absolue.
Si la place $v$ est associ\'ee \`a un plongement de $K$ dans $\bC$, on dit
qu'elle est archim\'edienne et on \'ecrit $v\mid \infty$.  Sinon, on dit qu'elle est
ultram\'etrique, et on \'ecrit $v\mid p$ si elle est associ\'ee \`a un
plongement de $K$ dans $\bC_p$.  Si $\alpha\in K$, alors la
s\'erie pour $e^\alpha$ converge dans chaque compl\'et\'e archim\'edien
de $K$ mais ne converge que dans un nombre fini de compl\'et\'es
ultram\'etriques.  En particulier,
si $K$ ne poss\`ede qu'une seule place archim\'edienne, c'est-\`a-dire si
$K=\bQ$ ou si $K$ est quadratique imaginaire, alors il arrive que
$e^\alpha$ n'ait un sens que pour cette place.  Dans ce cas, nous
obtenons l'estimation suivante o\`u $\cO_K$ d\'esigne l'anneau des
entiers de $K$.

\begin{proposition}
\label{intro:prop:imaginaire}
Soit $K\subset \bC$ le corps $\bQ$ ou une extension quadratique imaginaire
de $\bQ$, et soit $\alpha$ un \'el\'ement non nul de $K$ tel que
$|\alpha|_v\ge p^{-1/(p-1)}$ pour tout nombre premier $p$
et tout place $v$ de $K$ avec $v\mid p$.  Alors, pour tout choix
de $x,y\in\cO_K$ avec $x\neq 0$, on a
\[
 |x|\,|xe^\alpha-y| \ge c(\log |x|)^{-2g-1}
\]
o\`u $g$ d\'esigne le nombre de places $v$ de $K$ avec $v\mid \infty$
ou $|\alpha|_v\neq 1$ et o\`u $c>0$ est une constante qui ne d\'epend
que de $\alpha$ et de $K$.
\end{proposition}

Par exemple si $K=\bQ(\sqrt{-2})$, on peut prendre
$\alpha=2(1\pm\sqrt{-2})/m$ o\`u $m\in\cO_K\setminus\{0\}$.
Si $K=\bQ(\sqrt{-23})$, on peut prendre
$\alpha=(1\pm\sqrt{-23})/(2m)$ o\`u $m\in\cO_K\setminus\{0\}$.
Dans certains cas, $e^\alpha$ admet un d\'eveloppement en
fraction continue g\'en\'eralis\'e semblable \`a celui de $e$
(avec quotients partiels dans $\cO_K$) mais
nous ne consid\'erons pas cette question ici.

Plus g\'en\'eralement, soient $\alpha_1,\dots,\alpha_s$ des
\'el\'ements distincts d'un corps de nombres $K\subset \bC$.
Le th\'eor\`eme de Lindemann-Weierstrass \cite{We1885} nous apprend que
$e^{\alpha_1},\dots,e^{\alpha_s}\in\bC$ sont lin\'eairement
ind\'ependants sur $K$ et la preuve classique de ce r\'esultat,
dans toutes ses variantes (voir \cite[Appendix]{Ma1976}),
utilise les approximations d'Hermite
dont nous rappelons la d\'efinition au prochain paragraphe.  Le
but de ce travail est de montrer que ces approximations sont
quasiment optimales au sens de la g\'eom\'etrie des nombres
dans les ad\`eles de $K$, en tenant compte de toutes les places
$v$ de $K$ et de toutes les paires d'indices $i,j$ avec
$1\le i<j\le s$ telles que la s\'erie pour $e^{\alpha_i-\alpha_j}$
soit convergente dans $K_v$.  Il est possible que cette observation
soit le reflet d'une propri\'et\'e plus large des valeurs de la
fonction exponentielle.

Par exemple la s\'erie pour $e^3$ converge dans $\bR$ et dans $\bQ_3$
mais pas dans $\bQ_p$ quel que soit le nombre premier $p\neq 3$.
Alors notre approche conduit au r\'esultat suivant.

\begin{proposition}
\label{intro:prop:e3}
Pour tout entier $n\ge 1$, on d\'efinit un corps convexe $\cC_n$
de $\bR^2$ et un r\'eseau $\Lambda_n$ de $\bR^2$ par
\begin{align*}
 \cC_{n}
  &=
 \left\{ (x,y)\in\bR^2 \,;
   \quad
   |x|\le \frac{(2n)!}{n!3^{n/2}}
   \,,\quad 
   |xe^3-y|\le \Big(\frac{3}{2}\Big)^{2n}\frac{1}{n!3^{n/2}}
 \right\},\\
 \Lambda_{n}
  &=
 \left\{ (x,y)\in\bZ^2 \,;
   \quad
   |xe^3-y|_3\le 3^{-n}
 \right\}\,.
\end{align*}
Pour $i=1,2$, on note
$\lambda_i(\cC_n,\Lambda_n)$ le $i$-i\`eme minimum de $\cC_n$
relativement au r\'eseau $\Lambda_n$,c'est-\`a-dire le plus petit
$\lambda>0$ tel que $\lambda\cC_n$ contienne au moins $i$
\'el\'ements de $\Lambda_n$ lin\'eairement ind\'ependants
sur $\bQ$.  Alors on a
\[
 (cn^2)^{-1}
  \le \lambda_1(\cC_n,\Lambda_n)
  \le \lambda_2(\cC_n,\Lambda_n)
  \le cn^2,
\]
pour une constante $c>1$ ind\'ependante de $n$.
\end{proposition}

En utilisant le fait que $3^n\bZ^2\subset\Lambda_n$, on en d\'eduit
que $\lambda_1(\cC_n,\bZ^2)\ge (c n^2 3^n)^{-1}$ pour tout $n\ge 1$.
En cons\'equence, pour chaque $\epsilon>0$, il existe une
constante $c_\epsilon>0$ telle que
\[
 |x|\,|xe^{3}-y|
  \ge c_\epsilon |x|^{-\epsilon}
\]
pour tout $(x,y)\in\bZ^2$ avec $x\neq0$. On peut m\^eme faire un
peu mieux (voir \cite[Satz 1]{Bu1971}).  Or, des calculs
num\'eriques expliqu\'es au paragraphe \ref{sec:num} donnent
\begin{equation}
 \label{intro:eq:loglog}
 |x|\,|xe^3-y|\ge (3\,\log|x|\,\log\log |x|)^{-1}
 \quad\text{si}\quad 4\le |x|\le 10^{500\,000}.
\end{equation}
Des calculs plus laborieux que nous ne d\'ecrivons pas ici
sugg\`erent m\^eme qu'il existe un r\'eel $g>0$ tel que
\[
 |x_1|\,|x_1e^3-x_2|\,|x_1e^3-x_2|_3 \ge (\log |x_1|)^{-g}
\]
pour tout $(x_1,x_2,x_3)\in\bZ^3$ avec $|x_1|$ assez grand.
Enfin, un r\'esultat important de Baker \cite{Ba1965} montre que si
$\alpha_2,\dots,\alpha_s\in\bQ$ sont des nombres rationnels
distincts non nuls, alors, pour chaque $\epsilon>0$, il existe
aussi une constante $c_\epsilon>0$ telle que
\[
 |x_1|\,|x_1e^{\alpha_2}-x_2|\cdots|x_1e^{\alpha_s}-x_s|
  \ge c_\epsilon |x_1|^{-\epsilon}
\]
pour tout $(x_1,\dots,x_s)\in\bZ^s$ avec $|x_1|\neq 0$.  Les
propri\'et\'es des approximations d'Hermite sugg\`erent que
le membre de droite de cette in\'egalit\'e
$c_\epsilon |x_1|^{-\epsilon}$ pourrait \^etre remplac\'e par
$(\log |x_1|)^{-g}$ pour une constante $g>0$ qui ne d\'epend
que de $(\alpha_2,\dots,\alpha_s)$, lorsque $|x_1|$ est assez grand.

Dans cet article, $\bN$ d\'esigne l'ensemble des entiers
positifs ou nul et $\bN_+=\bN\setminus\{0\}$ l'ensemble
des entiers positifs.

\medskip
\noindent
\textbf{Remerciements:} L'auteur remercie chaleureusement
Michel Waldschmidt pour de nombreux \'echanges sur ces
questions. En particulier, ses notes de cours \cite{Wa2008}
sont une des sources d'inspiration de ce travail.

%
%

\section{\'Enonc\'e du r\'esultat principal}
\label{sec:resultat}

Soit $K$ un corps de nombres, soit $\cO_K$ son anneau
d'entiers, soit $d=[K:\bQ]$ son degr\'e sur $\bQ$, et
soit $s$ un entier positif.  On note $\cO_K$ l'anneau
des entiers de $K$.
Pour toute place ultram\'etrique $v$ de $K$, on d\'esigne par
$\cO_v=\{x\in K_v\,;\,|x|_v\le 1\}$ l'anneau des entiers de $K_v$
et par $d_v=[K_v:\bQ_p]$ le degr\'e local de $K_v$, o\`u $p$
d\'esigne le nombre premier en dessous de $v$ (notation $v\mid p$),
c'est-\`a-dire le nombre premier $p$ tel que $|\ |_v$ \'etende
la valeur absolue $p$-adique sur $\bQ$.
Suivant McFeat \cite[\S2.2]{Mc1971}, on note
$\mu_v$ la mesure de Haar sur $K_v$ normalis\'ee de telle sorte que
$\mu_v(\cO_v)=1$.  Pour une place archim\'edienne (notation
$v\mid \infty$), on note encore $d_v=[K_v:\bR]$ le degr\'e local de
$K_v$ et on d\'esigne par $\mu_v$ la mesure de Lebesgue sur $K_v$
(ce corps est $\bR$ ou $\bC$).  On note $r_1$ (resp.\ $r_2$)
le nombre de places $v\mid \infty$ avec $d_v=1$ (resp.\ $d_v=2$),
de sorte que $d=r_1+2r_2$.

L'anneau des ad\`eles de $K$ est le produit $K_\bA=\prod_v K_v$
\'etendu \`a toute les places $v$ de $K$, avec la topologie
restreinte.  C'est un anneau localement compact qu'on
munit de la mesure de Haar $\mu$, produit des $\mu_v$.
On voit $K$ comme un sous-corps de $K_\bA$ via le plongement
diagonal.  Alors $K$ est un sous-groupe discret de $K_\bA$ et,
avec la normalisation de $\mu$, on a
\[
 \mu(K_\bA/K)=2^{-r_2}|D(K)|^{1/2},
\]
o\`u $D(K)$ d\'esigne le discriminant de $K$.
Par abus de notation, on d\'esigne encore par $\mu$ la mesure
sur $K_\bA^s$, produit de $s$ copies de $\mu$. Pour toute place
$v$ de $K$, on d\'esigne aussi par $\mu_v$ la mesure sur $K_v^s$,
produit de $s$ copies de $\mu_v$.  Avec notre normalisation de
la valeur absolue sur $K_v$, si $T\colon K_v^s\to K_v^s$ est
une application $K_v$-lin\'eaire et si $E$ est un sous-ensemble
mesurable de $K_v^s$, alors $T(E)$ est mesurable de mesure
$\mu_v(T(E))=|\det T|_v^{d_v}\mu_v(E)$.

%
%

\subsection{Minima des convexes ad\'eliques}
\label{subsec:minima}
Un convexe ad\'elique de $K^s$ est un produit
\[
  \cC=\prod_v \cC_v \subset K_\bA^s
\]
portant sur toutes les places $v$ de $K$, qui poss\`ede
les propri\'et\'es suivantes:
\begin{itemize}
\item[(i)] si $v\mid \infty$, alors $\cC_v$ est un \emph{corps convexe}
  de $K_v^s$, c'est-\`a-dire un voisinage compact
  et convexe de $0$ dans $K_v^s$ tel que $\alpha\,\cC_v=\cC_v$ pour
  tout $\alpha\in K_v$ avec $|\alpha|_v=1$,
\item[(ii)] si $v\nmid\infty$, alors $\cC_v$ est un
  sous-$\cO_v$-module de $K_v^s$ de rang $s$ et de type fini (donc libre),
\item[(iii)] $\cC_v=\cO_v^s$ pour toutes sauf un nombre fini de
  places $v$ de $K$ avec $v\nmid \infty$.
\end{itemize}
Supposons donn\'e un tel produit $\cC$.  Pour chaque $i=1,\dots,s$,
on d\'efinit son $i$-i\`eme minimum $\lambda_i(\cC)$ comme \'etant
le plus petit $\lambda>0$ tel que le convexe ad\'elique
\[
 \lambda\cC=\prod_{v\mid \infty}\lambda\cC_v \prod_{v\nmid\infty}\cC_v
\]
contienne au moins $i$ \'el\'ements de $K^s$ lin\'eairement
ind\'ependants sur $K$. Avec ces notations et notre choix de
normalisation des mesures, la version ad\'elique du
th\'eor\`eme de Minkowski s'\'enonce ainsi.

\begin{theoreme}[McFeat, Bombieri et Vaaler]
 \label{res:thm:MBV}
Pour tout convexe ad\'elique $\cC$ de $K^s$, on a
\[
  2^{sr_1}(s!)^{-d}
  \le
  \left(\lambda_1(\cC)\cdots\lambda_s(\cC)\right)^d\mu(\cC)
  \le
  2^{s(r_1+r_2)}|D(K)|^{s/2}.
\]
\end{theoreme}

On renvoie le lecteur \`a \cite[Theorem 5]{Mc1971} et
\cite[Theorem 3]{BV1983} pour la borne sup\'erieure du
produit des minimas (voir aussi la majoration de Thunder dans
\cite[Theorem 1 and Corollary]{Th2002}).  La borne inf\'erieure
donn\'ee ici est celle de \cite[Theorem 6]{Mc1971},
un peu moins pr\'ecise que celle de \cite[Theorem 6]{BV1983}.

%
%

\subsection{Approximations d'Hermite}
\label{subsec:approx:Hermite}
Soient $\alpha_1,\dots,\alpha_s$ des \'el\'ements distincts de $K$.
Pour tout $s$-uplet $\un:=(n_1,\dots,n_s)\in\bN^s$, on d\'efinit
des polyn\^omes de $K[z]$ par
\[
 f_\un(z)=(z-\alpha_1)^{n_1}\cdots(z-\alpha_s)^{n_s}
\et
 P_\un(z)=\sum_{k=0}^N f_\un^{(k)}(z)
\]
o\`u
\[
 N=n_1+\cdots+n_s
\]
repr\'esente le degr\'e de $f_\un$, et o\`u
$f_\un^{(k)}$ d\'esigne la $k$-i\`eme d\'eriv\'ee de $f_\un$
pour chaque entier $k\ge 0$. Puis on forme le point
\[
 a_\un:=\big(P_\un(\alpha_1),\dots,P_\un(\alpha_s)\big)
   \in K^s.
\]
Nous l'appelerons \emph{l'approximation
d'Hermite d'ordre $\un$ associ\'ee au $s$-uplet
$(\alpha_1,\dots,\alpha_s)$}.  Le but de ce travail
est de donner un sens pr\'ecis au terme ``approximation'',
en travaillant sur les ad\`eles de $K$.

Rappelons d'abord quelques propri\'et\'es de ces
points. Pour simplifier, commen\c{c}ons par supposer
$K\subset \bC$.  On trouve
\begin{equation}
 \label{resultat:equadiff}
 \frac{d}{dz}\big( P_\un(z)e^{-z}\big)
  = \big( P_\un'(z)-P_\un(z) \big) e^{-z}
  =  -f_\un(z)e^{-z}.
\end{equation}
Donc, pour toute paire d'indices $i,j\in\{1,\dots,s\}$, on a
\begin{equation*}
 \label{resultat:integrale_finie}
 P_\un(\alpha_i)e^{-\alpha_i}-P_\un(\alpha_j)e^{-\alpha_j}
  = \int_{\alpha_i}^{\alpha_j} f_\un(z) e^{-z}\,dz\,,
\end{equation*}
l'int\'egrale \'etant prise sur n'importe quel chemin
de $\alpha_i$ \`a $\alpha_j$ dans $\bC$.  En int\'egrant
le long du segment de droite $[\alpha_i,\alpha_j]$ qui lie
ces deux points et en notant que
\[
 \max_{z\in [\alpha_i,\alpha_j]} |f_\un(z)|
 \le
 R^N
 \quad\text{avec}\quad
 R=\max_{1\le k,\ell\le s}|\alpha_k-\alpha_\ell|,
\]
on en d\'eduit que
\begin{equation*}
 \label{resultat:maj_diff}
 \left|
 P_\un(\alpha_i)e^{-\alpha_i}-P_\un(\alpha_j)e^{-\alpha_j}
 \right|
 \le
 c_1 R^N
\end{equation*}
pour une constante $c_1>0$ ind\'ependante du choix de $i$,
$j$ et $\un$.  De m\^eme, pour $i=1,\dots,s$, l'\'equation
\eqref{resultat:equadiff} livre
\begin{equation*}
 \label{resultat:integrale_infinie}
 P_\un(\alpha_i)
  = \int_0^\infty f_\un(z+\alpha_i) e^{-z}\,dz\,,
\end{equation*}
l'int\'egrale \'etant prise sur la demi-droite $[0,\infty)\subset \bR$.
Comme $|f_\un(t+\alpha_i)|\le (t+R)^N$ pour tout $t\ge 0$,
on en d\'eduit que
\begin{equation*}
 \label{resultat:maj_point}
 |P_\un(\alpha_i)|
  \le \int_0^\infty (t+R)^N e^{-t}\,dt
  \le e^R\int_0^\infty t^N e^{-t}\,dt
  = e^R N!\,.
\end{equation*}

Plus g\'en\'eralement, soit $v$ une place archim\'edienne de $K$.
Posons
\begin{equation}
 \label{resultat:eq:Rv}
 R_v=\max_{1\le k,\ell\le s}|\alpha_k-\alpha_\ell|_v
\end{equation}
et choisissons un plongement $\sigma\colon K\to\bC$
tel que $|\alpha|_v=|\sigma(\alpha)|$ pour tout
$\alpha\in K$.  Alors, pour toute paire d'indices
$i,j\in\{1,\dots,s\}$, les calculs ci-dessus livrent
\begin{align}
 \label{resultat:v:maj_diff}
 \left|
  P_\un(\alpha_i)e^{-\alpha_i}-P_\un(\alpha_j)e^{-\alpha_j}
 \right|_v
 &=
 \left|
  \int_{\sigma(\alpha_i)}^{\sigma(\alpha_j)} f^\sigma_\un(z) e^{-z}\,dz
 \right|
 \le c_v R_v^N,\\
 \label{resultat:v:maj_point}
 \left| P_\un(\alpha_i) \right|_v
 &\le
  e^{R_v} N!\,,
\end{align}
o\`u $f^\sigma_\un$ d\'esigne l'image de $f_\un$ sous
l'homomorphisme d'anneaux de $K[z]$ dans $\bC[z]$ qui fixe $z$
et \'etend $\sigma$, et o\`u $c_v>0$ ne d\'epend que de $v$
et de $\alpha_1,\dots,\alpha_s$.  Ainsi, $a_\un$ est une
approximation projective de $(e^{\alpha_1},\dots,e^{\alpha_s})$
en chaque place archim\'edienne de $K$.

Dans ce travail, nous \'etablissons une majoration de
l'int\'egrale de \eqref{resultat:v:maj_diff} qui est plus
fine que $c_vR_v^N$ pour chaque place archim\'edienne $v$
de $K$.  Nous donnons aussi des analogues de
\eqref{resultat:v:maj_diff} et de \eqref{resultat:v:maj_point}
pour les places ultram\'etriques $v$ de $K$ chaque fois que
leur membre de gauche poss\`ede un sens dans $K_v$.
Plus pr\'ecis\'ement, comme $e^{\alpha_j-\alpha_i}$ peut
avoir un sens dans $K_v$ sans que $e^{\alpha_i}$ et
$e^{\alpha_j}$ en aient un, nous consid\'erons plut\^ot
les quantit\'es
$|P_\un(\alpha_i)e^{\alpha_j-\alpha_i}-P_\un(\alpha_j)|_v$.
Ici encore, nous aurons besoin d'estimations fines alors que
d'habitude on s'affranchit de mani\`ere exp\'editive des estimations
ultram\'etriques.  En g\'en\'eral, on choisit un d\'enominateur
commun $b$ de $\alpha_1,\dots,\alpha_s$, c'est-\`a-dire un
entier $b\ge 1$ tel que
$b\alpha_1,\dots,b\alpha_s\in \cO_K$.  Alors le polyn\^ome
$g(t):=b^Nf(t/b)$ est \`a coefficients dans $\cO_K$ et, pour tout
$i=1,\dots,s$, on trouve
\begin{equation*}
 \label{resultat:denom}
 \frac{b^N}{(n_i)!}P_\un(\alpha_i)
 = \sum_{k=n_i}^N \frac{b^N}{(n_i)!} f^{(k)}(\alpha_i)
 = \sum_{k=n_i}^N \frac{b^k k!}{(n_i)!} \cdot \frac{g^{(k)}(b\alpha_i)}{k!}
 \in \cO_K\,.
\end{equation*}
Par exemple, si $n_1=\cdots=n_s=n$, cela implique que $(b^N/n!)a_\un\in\cO_K^s$.

Les estimations ci-dessus sont des ingr\'edients-cl\'es dans la preuve
classique du th\'eor\`eme de Lindemann-Weiertrass visant \`a montrer que
$e^{\alpha_1},\dots,e^{\alpha_s}$ sont lin\'eairement ind\'ependants sur $K$.
Il en manque cependant encore deux.  Le premier est une astuce de r\'eduction
de Weierstrass qui est expliqu\'ee dans \cite[Appendix, \S3]{Ma1976}
(voir aussi \cite[Chapitre 1, \S3]{Ba1975}).  Le second est
l'existence de familles de $s$ approximations lin\'eairement ind\'ependantes
sur $K$. Hermite lui-m\^eme avait soulign\'e et r\'esolu cette
difficult\'e pour \'etablir la transcendance de $e$.  Nous
utiliserons ici le remarquable r\'esultat de Mahler que voici.

\begin{theoreme}[Mahler]
\label{res:thm:Mahler}
Supposons que $\un=(n_1,\dots,n_s)\in\bN_+^s$ ait toutes ses
coordonn\'ees positives.  On note $\ue_1=(1,0,\dots,0),\dots,
\ue_s=(0,\dots,0,1)$ les \'el\'ements de la base canonique de
$\bZ^s$.  Alors, on a
\begin{equation}
 \label{resultat:det:Mahler}
 \Delta_\un
 :=
 \det(a_{\un-\ue_1},\dots,a_{\un-\ue_s})
 =
 \prod_{i=1}^s
     \left(
     (n_i-1)!\prod_{k\neq i} (\alpha_i-\alpha_k)^{n_k}
     \right)
 \neq 0.
\end{equation}
\end{theoreme}

La preuve de Mahler est astucieuse.  Elle est donn\'ee
dans \cite[\S8]{Ma1932} et reprise
dans \cite[Appendix, \S16]{Ma1976}.  Dans le cas o\`u
$n_1=\dots=n_s$, ce r\'esultat est d\^u \`a Hermite
\cite{He1873}.  La preuve d'Hermite est diff\'erente
et utilise les relations de r\'ecurrence que satisfont
les points $\ua_\un$ et que nous rappelons dans l'appendice
\ref{sec:rel}.

%
%

\subsection{\'Enonc\'e du r\'esultat principal}
\label{subsec:res}
Avec les notations pr\'ec\'edantes, on note $E$ l'ensemble
fini constitu\'e des places archim\'ediennes de $K$
ainsi que des places ultram\'etriques $v$ de $K$ telles que
$|\alpha_i-\alpha_j|_v\neq 1$ pour au moins un couple d'entiers
$i,j\in\{1,\dots,s\}$ avec $i\neq j$.  Pour chaque $s$-uplet
$\un=(n_1,\dots,n_s)\in\bN_+^s$, on note $N$ sa somme et
on d\'efinit un convexe ad\'elique $\cC_\un=\prod_v\cC_{\un,v}$
de $K^s$ de la mani\`ere suivante.
\begin{itemize}
\item[(i)] Si $v\,|\,\infty$ est la place associ\'ee \`a un
 plongement $\sigma\colon K\hookrightarrow\bC$, on d\'efinit
 $R_v$ par \eqref{resultat:eq:Rv}. Alors $\cC_{\un,v}$ est
 l'ensemble des points $(x_1,\dots,x_s)\in K_v^s$ qui satisfont
\begin{equation}
 \label{res:eq:C_arch}
 |x_i|_v\le e^{R_v} (N-1)!
 \et
 |x_ie^{\alpha_j-\alpha_i}-x_j|_v
 \le \max_{1\le k\le s}
     \left|
       \int_{\sigma(\alpha_i)}^{\sigma(\alpha_j)}
          f_{\un-\ue_k}^\sigma(z)e^{\sigma(\alpha_j)-z}\,dz
     \right|
\end{equation}
pour chaque paire d'entiers $i,j\in\{1,\dots,s\}$ avec $i\neq j$.
\smallskip
\item[(ii)]
Si $v\in E$ et si $v\,|\,p$ pour un nombre premier $p$,
alors $\cC_{\un,v}$ est l'ensemble
des points $(x_1,\dots,x_s)\in K_v^s$ qui satisfont
\begin{equation}
 \label{res:eq:ultra1}
 |x_i|_v
  \le p^3 N
      \prod_{1\le k\le s}
      \max\big\{|\alpha_i-\alpha_k|_v,\,p^{-1/(p-1)}\big\}^{n_k}
\end{equation}
pour $i=1,\dots,s$, ainsi que
\begin{equation}
 \label{res:eq:ultra2}
 |x_ie^{\alpha_j-\alpha_i}-x_j|_v
 \le p^3 N
      \prod_{1\le k\le s}
      \max\big\{|\alpha_i-\alpha_k|_v,|\alpha_j-\alpha_k|_v\big\}^{n_k}.
\end{equation}
pour chaque paire d'entiers $i,j\in\{1,\dots,s\}$ avec
$0 < |\alpha_j-\alpha_i|_v < p^{-1/(p-1)}$.
\smallskip
\item[(iii)] Enfin, si $v\notin E$, alors $\cC_{\un,v}$ est l'ensemble
des points $(x_1,\dots,x_s)\in K_v^s$ qui satisfont
\[
 |x_i|_v \le |(n_i-1)!|_v
\]
pour $i=1,\dots,s$.
\end{itemize}
La particularit\'e cruciale de ces convexes ad\'eliques $\cC_\un$
est que les formes lin\'eaires qui les d\'efinissent ne font
intervenir que les valeurs complexes ou $p$-adiques de la
fonction exponentielle en les nombres $\alpha_i$ ou
$\alpha_j-\alpha_i$.  Comme le montrent les estimations du paragraphe
\S\ref{subsec:approx:Hermite}, leur composante $\cC_{\un,v}$
contient les points $a_{\un-\ue_1},\dots,a_{\un-\ue_s}$
pour toute place archim\'edienne $v$ de $K$.  On va montrer au
paragraphe suivant que cela vaut en fait pour toutes les
places de $K$, d'o\`u la premi\`ere assertion de l'\'enonc\'e
suivant.

\begin{theoreme}
\label{res:principal}
Soit $\un=(n_1,\dots,n_s)\in\bN_+^s$. Alors le convexe ad\'elique
$\cC_\un$ contient les points $a_{\un-\ue_1},\dots,a_{\un-\ue_s}$.
De plus, en posant $N=n_1+\cdots+n_s$,
on a les estimations suivantes.
\begin{itemize}
\item[(i)] Si $v\mid \infty$, alors
 \[
   (s!)^{-1} |\Delta_\un|_v
    \le \mu_v(\cC_{\un,v})^{1/d_v}
    \le c_v N^{2s-2} |\Delta_\un|_v
 \]
 pour une constante $c_v>0$ qui ne d\'epend que
 $\alpha_1,\dots,\alpha_s$ et de $v$.
\smallskip
\item[(ii)] Si $v\in E$ et si $v\mid p$ pour un nombre premier $p$, alors
 \[
   |\Delta_\un|_v\le \mu_v(\cC_{\un,v})^{1/d_v} \le (p^3N)^s |\Delta_\un|_v.
 \]
\item[(iii)] Si $v\notin E$, alors $\mu_v(\cC_{\un,v})^{1/d_v}=|\Delta_\un|_v$.
\end{itemize}
\end{theoreme}

On remarquera que, pour toutes les places $v$ de $K$, ces
estimations enferment le volume de $\cC_{\un,v}$ entre deux bornes
dont le rapport est un polyn\^ome en $N$ alors que les bornes
elles m\^emes croissent comme $|\Delta_\un|_v$ c'est-\`a-dire
en gros comme une exponentielle en $N$ si $v\nmid\infty$ et m\^eme
comme $N!$ si $v\mid \infty$.  Pour $v\mid \infty$, on donne une valeur
explicite de la constante $c_v$ au th\'eor\`eme \ref{volarch:thm}.

Les bornes inf\'erieures pour $\mu_v(\cC_{\un,v})$ d\'ecoulent
ais\'ement de la d\'efinition de $\Delta_\un$ donn\'ee
au th\'eor\`eme \ref{res:thm:Mahler},
si on accepte le fait que $\cC_{\un,v}$ contient les
points $\ua_{\un-\ue_i}$ pour $i=1,\dots,s$.  En effet, soit
$T\colon K_v^s\to K_v^s$ l'application $K_v$-lin\'eaire donn\'ee
par
\[
 T(x_1,\dots,x_s)=x_1\ua_{\un-\ue_1}+\cdots+x_s\ua_{\un-\ue_s}
\]
pour tout $(x_1,\dots,x_s)\in K_v^s$.  Alors
$\cC_{\un,v}$ contient $T(\cE_v)$ o\`u $\cE_v$ est donn\'e par
\begin{align*}
  \cE_v &=\{(x_1,\dots,x_s)\in K_v^s\,;\,|x_1|_v+\cdots+|x_s|_v\le 1\}
  &&\text{si $v\mid \infty$,}\\
  \cE_v &=\cO_v^s
  &&\text{si $v\nmid\infty$.}
\end{align*}
Comme $|\det T|_v=|\Delta_\un|_v$,
on a $\mu_v(T(\cE_v)) =|\Delta_\un|_v^{d_v}\mu_v(\cE_v)$.  Si $v\mid \infty$,
on a aussi $\mu_v(\cE_v)\ge (s!)^{-d_v}$, donc $\mu_v(\cC_{\un,v})^{1/d_v}
\ge (s!)^{-1}|\Delta_\un|_v$.  Si $v\nmid\infty$, on a simplement
$\mu_v(\cE_v)=1$, donc $\mu_v(\cC_{\un,v})^{1/d_v} \ge |\Delta_\un|_v$.

L'essentiel du travail concerne donc les majorations de volume
des $\cC_{\un,v}$.  On en donne un aper\c{c}u plus bas.  Elles
conduisent \`a une majoration du volume de $\cC_\un$ dont on
tire la conclusion suivante gr\^ace au th\'eor\`eme de Minkowski
ad\'elique.

\begin{cor}
\label{res:cor}
Avec les notations du th\'eor\`eme \ref{res:principal}, on a
\[
 cN^{-g}\le \lambda_1(\cC_\un)\le\cdots\le\lambda_s(\cC_\un)\le 1
 \quad
 \text{avec}\quad
 g=s-2+s\sum_{v\in E} \frac{d_v}{d},
\]
o\`u $c>0$ est une constante qui ne d\'epend que de $\alpha_1,\dots,\alpha_s$.
\end{cor}

\begin{proof}[D\'emonstration]
Comme $\prod_v |\Delta_\un|_v^{d_v} = 1$ et que $E$ contient toutes
les places archim\'ediennes de $K$, on trouve
\[
 \mu(\cC_n)
  = \prod_v \mu_v(\cC_{\un,v})
  \le \prod_{v\mid \infty} \Big(c_v^{d_v}N^{(2s-2)d_v}\Big)
      \prod_{v\in E,\,v\mid p} (p^3N)^{sd_v}
  = c_1^d N^{gd}
\]
o\`u $c_1>0$ est une constante ind\'ependante de $\un$.
Comme $\cC_\un$ contient les points $a_{\un-\ue_1},\dots,a_{\un-\ue_s}$
de $K^s$ et qu'en vertu du th\'eor\`eme \ref{res:thm:Mahler} ceux-ci
sont lin\'eairement ind\'ependants sur $K$, on a aussi
\[
  \lambda_1(\cC_\un)\le\cdots\le\lambda_s(\cC_\un)\le 1.
\]
Alors en vertu du th\'eor\`eme \ref{res:thm:MBV}, on a
\[
 (s!)^{-1}
  \le \lambda_1(\cC_\un)\cdots\lambda_s(\cC_\un)\mu(\cC_\un)^{1/d}
  \le \lambda_1(\cC_\un)c_1N^g,
\]
donc $\lambda_1(\cC_\un)\ge cN^{-g}$ avec $c=1/(c_1s!)$.
\end{proof}

La preuve du th\'eor\`eme \ref{res:principal} utilise des
r\'esultats g\'en\'eraux sur les polyn\^omes $f(z)\in\bC[z]$
\`a coefficients complexes que nous n'avons pas trouv\'es
dans la litt\'erature.  Supposons $f$ de degr\'e $N\ge 1$.
Soit $A$ l'ensemble de ses racines dans $\bC$ et soit $B$ l'ensemble
des racines de sa d\'eriv\'ee $f'$ qui ne sont pas dans $A$.
Au paragraphe \ref{sec:descente}, on consid\`ere les chemins
de descente maximale pour $|f|$ issus d'un point quelconque
$\beta$ de $\bC$.  Ceux-ci terminent n\'ecessairement
sur un \'el\'ement de $A$. On montre qu'ils sont contenus
dans l'enveloppe convexe de $A\cup\{\beta\}$, de longueur
au plus $\pi RN$ o\`u $R$ est le rayon d'un disque qui
contient $A\cup\{\beta\}$.  Au paragraphe \ref{sec:arch},
pour chaque $\beta\in B$, on note $m(\beta)$ la multiplicit\'e
de $\beta$ comme racine de $f'$ et, partant du point $\beta$,
on choisit $m(\beta)+1$ chemins de descente pour $|f|$ qui sont
localement distincts au voisinage de $\beta$.  Ces chemins
dessinent un graphe sur $A\cup B$ et on montre que ce graphe
est en fait un arbre.  On en extrait un sous-graphe $G$ sur $A$
qui est aussi un arbre avec ses ar\^etes index\'ees par $B$.
Pour chaque ar\^ete de $G$ d'extr\'emit\'es $\alpha,\alpha'
\in A$,  associ\'ee \`a un point $\beta\in B$, on dispose
d'un chemin qui joint $\alpha$ \`a $\alpha'$ en passant par
$\beta$, de longueur contr\^ol\'ee, le long duquel $|f|$ est
maximal au point $\beta$.

Pour la preuve du th\'eor\`eme \ref{res:principal} (i), on
peut supposer que la place $v\mid \infty$ donn\'ee provient
d'une inclusion $K\subset \bC$ et on applique la construction
d\'ecrite ci-dessus en prenant pour $f$ le pgcd des
polyn\^omes
$f_{\un-\ue_1},\dots,f_{\un-\ue_s}$ pour le
choix de $\un\in\bN_+^s$ donn\'e.  Si les coordonn\'ees de
$\un$ sont toutes $\ge 2$, on obtient ainsi un arbre $G$ sur
$A=\{\alpha_1,\dots,\alpha_s\}$.  Alors, pour chaque ar\^ete de
$G$ d'extr\'emit\'es $\alpha_i,\alpha_j$, on peut majorer
les int\'egrales qui apparaissent dans \eqref{res:eq:C_arch}
en fonction de $|f(\beta)|$ o\`u $\beta\notin A$ est le z\'ero
de $f'$ associ\'e \`a l'ar\^ete.  On en d\'eduit au
paragraphe \ref{sec:volarch} une majoration du
volume du convexe $\cC_{\un,v}$ en fonction du produit
des $|f(\beta)|^{m(\beta)}$, lequel est le
semi-r\'esultant de $f$ et de $f'$.  La majoration annonc\'ee
pour $\mu_v(\cC_{\un,v})$ s'ensuit gr\^ace au calcul de ce
semi-r\'esultant (paragraphe \ref{sec:semi-resultant}).  Le
cas g\'en\'eral o\`u au moins une coordonn\'ee de $\un$ est
\'egale \`a $1$ demande un petit ajustement.

Le traitement des places ultram\'etriques $v\nmid\infty$ est
plus simple.  Au paragraphe \ref{sec:estultra},
on montre que $\cC_{\un,v}$ contient les points
$\ua_{\un-\ue_1},\dots,\ua_{\un-\ue_s}$.  Plus loin,
au paragraphe \ref{sec:ultra}, on cons\-truit une for\^et sur
$\{\alpha_1,\dots,\alpha_s\}$ associ\'ee \`a la place $v$.
Ce graphe permet de s\'electionner $s$ in\'egalit\'es parmi
\eqref{res:eq:ultra1} et \eqref{res:eq:ultra2}.  Au
paragraphe \ref{sec:volultra}, on en d\'eduit la majoration
annonc\'ee pour le volume de $\cC_{\un,v}$.
Les notions utiles de th\'eorie des graphes sont rappel\'ees
au paragraphe \ref{sec:graphes}.

Au paragraphe \ref{sec:dexp}, on revisite notre r\'esultat principal
pour les approximations ``diagonales'' de deux exponentielles,
c'est-\`a-dire pour $s=2$ et $n_1=n_2$.  On en donne
une forme plus fine dont la preuve utilise seulement les
estimations des paragraphes \ref{subsec:approx:Hermite} et \ref{sec:estultra}.
On s'en sert ensuite ensuite pour d\'emontrer les propositions
\ref{intro:prop:imaginaire} et \ref{intro:prop:e3} de l'introduction.

Pour conclure, on explique au paragraphe \ref{sec:num} comment
les formules de r\'ecurrence d'Hermite rappel\'ees dans
l'appendice \ref{sec:rel} permettent de calculer efficacement
les quotients partiels du d\'eveloppement en fraction continue
de $e^3$.  Cela permet de v\'erifier les in\'egalit\'es
\eqref{intro:eq:loglog} en moins de deux heures de
calculs sur un petit ordinateur.

%
%

\section{Estimations ultram\'etriques}
\label{sec:estultra}

Soit $v$ une place de $K$ au-dessus d'un nombre
premier $p$.  Le but de ce paragraphe est de
compl\'eter la preuve de la premi\`ere assertion du
th\'eor\`eme \ref{res:principal} en montrant
que la composante $\cC_{\un,v}$ de $\cC_\un$ contient les points
$\ua_{\un-\ue_1},\dots,\ua_{\un-\ue_s}$ pour tout $\un\in\bN_+^s$.
Pour y parvenir, on utilise les notations et r\'esultats
suivants.

Pour tout $a\in\bC_p$ et tout $r>0$, on note
\[
 B(a,r)=\{z\in\bC_p\,;\, |z-a|_p\le r\}
\]
le disque ferm\'e de $\bC_p$ de centre $a$ et de rayon $r$
(\`a la fois un ferm\'e et un ouvert de $\bC_p$).  Pour un tel disque
$B=B(a,r)$ et pour toute fonction analytique $g\colon B\to \bC_p$,
on pose
\[
 |g|_B=\sup\{|g(z)|_p\,;\, z\in B\}.
\]
Cette quantit\'e s'exprime aussi en fonction du d\'eveloppement en
s\'erie de Taylor de $g$ autour du point $a$ via la formule
\[
 |g|_B=\sup_{k\in\bN} \left|\frac{g^{(k)}(a)}{k!}\right|_p r^k,
\]
dont on tire la forme $p$-adique des in\'egalit\'es de Cauchy
\[
 |g^{(k)}(a)|_p \le |k!|_p r^{-k} |g|_B
 \quad
 (k\in\bN)
\]
(voir \cite[\S1.5]{Ro1978}).  Pour les calculs, on utilise aussi les
estimations
\begin{equation}
 \label{estultra:eq:delta}
 \delta^k\le |k!|_p \le k\delta^{k-p} \le p^2k\delta^k
 \quad
 (k\in \bN),
 \quad
 \text{o\`u}
 \quad
 \delta=p^{-1/(p-1)},
\end{equation}
qui d\'ecoulent de la formule $|k!|_p=p^{-m}$ o\`u
$m=\sum_{\ell=1}^\infty \lfloor k/p^\ell \rfloor$.

\begin{lemme}
 \label{estultra:lemme1}
Soit $\un=(n_1,\dots,n_s)\in\bN^s$, soit $N=n_1+\cdots+n_s$,
et soient $i,j\in\{1,\dots,s\}$.  Alors, on a
\begin{equation}
 \label{estultra:lemme1:eq1}
 |P_\un(\alpha_i)|_v
  \le p^2 N \prod_{k=1}^s \max\{|\alpha_i-\alpha_k|_v,\,\delta\}^{n_k}.
\end{equation}
Si $|\alpha_i-\alpha_k|_v\le 1$ pour $k=1,\dots,s$, on a aussi
\begin{equation}
 \label{estultra:lemme1:eq1bis}
 |P_\un(\alpha_i)|_v
  \le |n_i!|_v.
\end{equation}
Enfin, si $\rho=|\alpha_i-\alpha_j|_v$ satisfait $0<\rho<\delta$,
on a
\begin{equation}
 \label{estultra:lemme1:eq2}
 |P_\un(\alpha_i)e^{\alpha_j-\alpha_i}-P_\un(\alpha_j)|_v
  \le \frac{\rho}{\delta} p^2  N \prod_{k=1}^s
      \max\{|\alpha_i-\alpha_k|_v,\,|\alpha_j-\alpha_k|_v\}^{n_k}.
\end{equation}
\end{lemme}

\begin{proof}[D\'emonstration]
Pour simplifier, on peut supposer $K\subset \bC_p$ et
$|\alpha|_v=|\alpha|_p$ pour tout $\alpha\in K$.  Alors, on peut
voir le polyn\^ome $f_\un(z)\in K[z]$ comme une fonction analytique
$f_\un\colon\bC_p\to\bC_p$.  Pour estimer $|P_\un(\alpha_i)|_v
= |P_\un(\alpha_i)|_p$, on pose
\[
 B=B(\alpha_i,\delta)
 \et
 M=|f_\un|_B.
\]
Pour $k=0,1,\dots,N$, les in\'egalit\'es de Cauchy jointes \`a
\eqref{estultra:eq:delta} livrent
\[
 |f_\un^{(k)}(\alpha_i)|_p
  \le |k!|_p \delta^{-k} M
  \le p^2 k M
  \le p^2 N M,
\]
donc
\[
 |P_\un(\alpha_i)|_v
  = \left| \sum_{k=0}^N f_\un^{(k)}(\alpha_i)\right|_p
  \le p^2 N M.
\]
On en d\'eduit l'estimation \eqref{estultra:lemme1:eq1} car
\[
 M\le \prod_{k=1}^s \sup\{|z-\alpha_k|_p\,;\,z\in B\}^{n_k}
    = \prod_{k=1}^s \max\{|\alpha_i-\alpha_k|_v,\,\delta\}^{n_k}.
\]
Si $|\alpha_i-\alpha_k|_v\le 1$ pour tout $k$,
le m\^eme calcul fournit $|f_\un|_B\le 1$ en
prenant $B=B(\alpha_i,1)$.  Alors les in\'egalit\'es de Cauchy
donnent $|f_\un^{(k)}(\alpha_i)|_p\le |k!|_p$ pour tout $k\in\bN$.
Comme on a $f_\un^{(k)}(\alpha_i)=0$ pour $k=0,\dots,n_i-1$, on en
d\'eduit $|f_\un^{(k)}(\alpha_i)|_v\le |n_i!|_v$ pour tout $k\in\bN$
et l'estimation \eqref{estultra:lemme1:eq1bis} s'ensuit.

Supposons enfin que $0< \rho = |\alpha_i-\alpha_j|_p < \delta$.  Pour
d\'emontrer \eqref{estultra:lemme1:eq2}, on utilise cette fois
\[
 B = B(\alpha_j,\rho)
 \et
 M=|f_\un|_B.
\]
Comme $\rho<\delta$, la fonction $g\colon B\to \bC_p$ donn\'ee par
\[
 g(z)=P_\un(z)e^{\alpha_j-z}-P_\un(\alpha_j)
 \quad (z\in B)
\]
est analytique avec $ g(\alpha_j)=0$ et
\begin{equation}
 \label{estultra:lemme1:eq3}
 g'(z)=-f_\un(z)e^{\alpha_j-z}
 \quad (z\in B).
\end{equation}
Pour tout entier $\ell=0,1,\dots,N$, on a
\[
 |f_\un^{(\ell)}(\alpha_j)|_p
  \le |\ell!|_p \rho^{-\ell} M
  \le p^2 \ell (\delta/\rho)^\ell M
  \le p^2 N (\delta/\rho)^\ell M.
\]
Comme $\disp f_\un^{(\ell)}=0$ pour $\ell>N$, cette estimation s'\'etend
\`a tout $\ell\in\bN$. Alors, gr\^ace \`a \eqref{estultra:lemme1:eq3},
la formule de Leibniz pour la d\'eriv\'ee d'un produit livre,
pour tout entier $k\ge 1$,
\[
 |g^{(k)}(\alpha_j)|_p
  \le \max_{0\le \ell <k} |f_\un^{(\ell)}(\alpha_j)|_p
  \le p^2 N (\delta/\rho)^{k-1} M.
\]
Comme $\alpha_i\in B$ et que $g(\alpha_j)=0$, on en d\'eduit
\[
 |P_\un(\alpha_i)e^{\alpha_j-\alpha_i}-P_\un(\alpha_j)|_v
  =|g(\alpha_i)|_p
  \le |g|_B
   =  \sup_{k\ge 1} \left|\frac{g^{(k)}(\alpha_j)}{k!}\right|_p
       \rho^k
  \le p^2 N (\rho/\delta) M.
\]
L'estimation \eqref{estultra:lemme1:eq2} s'ensuit car
\[
 M\le \prod_{k=1}^s \sup\{|z-\alpha_k|_p\,;\,z\in B\}^{n_k}
    = \prod_{k=1}^s \max\{|\alpha_i-\alpha_k|_v,\,|\alpha_j-\alpha_k|_v\}^{n_k}.
\qedhere
\]
\end{proof}

\begin{theoreme}
Soit $\un=(n_1,\dots,n_s)\in\bN_+^s$.  Alors
le sous-$\cO_v$-module $\cC_{\un,v}$ de $K_v^s$ d\'efini
au paragraphe \ref{subsec:res} contient les
points $\ua_{\un-\ue_1},\dots,\ua_{\un-\ue_s}$.
\end{theoreme}

\begin{proof}[D\'emonstration]
Fixons un entier $\ell\in\{1,\dots,s\}$ quelconque et posons
$P=P_{\un-\ue_\ell}$.  Pour montrer que $\cC_{\un,v}$ contient
le point $\ua_{\un-\ue_\ell}= (P(\alpha_1),\dots,P(\alpha_s))$,
on choisit $i,j\in\{1,\dots,s\}$ quelconques.  Comme
$\max\{|\alpha_i-\alpha_\ell|_v,\,\delta\}\ge \delta\ge 1/p$,
l'in\'egalit\'e \eqref{estultra:lemme1:eq1} du
lemme \ref{estultra:lemme1} livre
\[
 \big|P(\alpha_i)\big|_v
 \le p^2(N-1) \frac{1}{\delta} \prod_{k=1}^s
     \max\{|\alpha_i-\alpha_k|_v,\,\delta\}^{n_k}
 \le p^3 N \prod_{k=1}^s \max\{|\alpha_i-\alpha_k|_v,\,\delta\}^{n_k}.
\]
Si $|\alpha_i-\alpha_k|_v=1$ pour tout $k=1,\dots,s$ avec $k\neq i$,
l'in\'egalit\'e \eqref{estultra:lemme1:eq1bis} du m\^eme lemme donne
aussi
\[
 \big|P(\alpha_i)\big|_v\le |(n_i-1)!|_v.
\]
Enfin, si $\rho=|\alpha_j-\alpha_i|_v$ satisfait $0<\rho<\delta$, alors,
comme $\max\{|\alpha_i-\alpha_\ell|_v,\,|\alpha_j-\alpha_\ell|_v\}
\ge \rho$, l'in\'egalit\'e \eqref{estultra:lemme1:eq2} du lemme livre
\begin{align*}
 |P(\alpha_i)e^{\alpha_j-\alpha_i}-P(\alpha_j)|_v
  &\le \frac{\rho}{\delta} p^2  (N-1)\cdot \frac{1}{\rho} \prod_{k=1}^s
      \max\{|\alpha_i-\alpha_k|_v,\,|\alpha_j-\alpha_k|_v\}^{n_k}\\
  &\le  p^3 N \prod_{k=1}^s
      \max\{|\alpha_i-\alpha_k|_v,\,|\alpha_j-\alpha_k|_v\}^{n_k}.
\qedhere
\end{align*}
\end{proof}

%
%

\section{Pr\'eliminaires de th\'eorie des graphes}
\label{sec:graphes}

Un \emph{graphe} $\cG$ est un couple d'ensembles finis $(V,E)$
o\`u $E$ consiste de sous-ensembles de $V$ avec deux \'el\'ements.
Les \'el\'ements de $V$ sont appel\'es les \emph{sommets}
de $\cG$ et ceux de $E$ les \emph{ar\^etes} de $\cG$ en accord
avec la repr\'esentation graphique usuelle.

Soit $\cG=(V,E)$ un graphe.  Une \emph{cha\^{\i}ne \'el\'ementaire}
dans $\cG$ est une suite $(\alpha_1,\dots,\alpha_m)$ d'\'el\'ements
distincts de $V$ avec $m\ge 2$ telle que
$\{\alpha_i,\alpha_{i+1}\}\in E$ pour $i=1,\dots,m-1$.  On dit
que $\cG$ est \emph{connexe} si, pour toute paire d'\'el\'ements
distincts $\alpha,\beta$ de $V$, il existe au moins une
cha\^{\i}ne \'el\'ementaire $(\alpha_1,\dots,\alpha_m)$
dans $\cG$ avec $\alpha_1=\alpha$ et $\alpha_m=\beta$.
On dit que $\cG$ est un \emph{arbre} s'il existe exactement
une telle cha\^{\i}ne pour chaque choix de $\alpha,\beta\in V$
avec $\alpha\neq \beta$.
Lorsque $\cG$ est connexe, on a $|V|\le |E|+1$ avec \'egalit\'e
si et seulement si $\cG$ est un arbre.

En g\'en\'eral, pour un graphe $\cG=(V,E)$, il existe un et
un seul choix d'entier $r\ge 1$ et de partitions
$V = V_1\cup\cdots\cup V_r$ et $E = E_1\cup\cdots\cup E_r$
de $V$ et $E$ en $r$ sous-ensembles disjoints tel que
$\cG_i=(V_i,E_i)$ soit un graphe connexe pour $i=1,\dots,r$.
On dit que $\cG_1,\dots,\cG_r$ sont les \emph{composantes
connexes} de $\cG$.  Si celles-ci sont des arbres, on dit
que $\cG$ est une \emph{for\^et}.
Lorsque $\cG$ poss\`ede $r$ composantes connexes, on a
$|V|\le |E|+r$ avec \'egalit\'e si et seulement si $\cG$
est une for\^et.

Une \emph{for\^et enracin\'ee} est un triplet $\cG=(R,V,E)$
o\`u $(V,E)$ est une for\^et et o\`u $R$ est un sous-ensemble
de $V$ qui contient exactement un sommet de chacune des
composantes connexes de $(V,E)$.
On dit que $R$ est l'ensemble des \emph{racines}
de $\cG$.  Alors, pour tout $\beta\in V\setminus R$, il existe
une et une seule cha\^{\i}ne \'el\'ementaire
$(\alpha_1,\dots,\alpha_m)$ avec $\alpha_1\in R$ et
$\alpha_m=\beta$.  Cela permet de d\'efinir un ordre partiel
sur $V$ en posant $\alpha<\beta$ si $\beta\notin R\cup\{\alpha\}$
et si la cha\^{\i}ne qui lie $\beta$ \`a un \'el\'ement de $R$
contient $\alpha$.  En particulier, toute ar\^ete $\{\alpha,\beta\}$
de $G$ peut \^etre ordonn\'ee de sorte que $\alpha<\beta$.
Les couples $(\alpha,\beta)$ ainsi form\'es sont appel\'es
les \emph{arcs} de $G$. Pour $\alpha\in V$ fix\'e, on dit que
$D_\cG(\alpha)=\{\beta\in V\,;\, \alpha<\beta\}$ est l'ensemble
des \emph{descendants} de $\alpha$. L'ensemble $S_\cG(\alpha)$
des \'el\'ements minimaux de $D_\cG(\alpha)$ est appel\'e
l'ensemble des \emph{successeurs} de $\alpha$.  Les couples
$(\alpha,\beta)\in V\times V$ avec
$\beta\in S_\cG(\alpha)$ sont donc les arcs de $G$.
De plus, tout $\beta\in V\setminus R$ est le successeur d'un et un
seul $\alpha\in V$.  Cela permet de formuler le r\'esultat suivant.

\begin{proposition}
\label{graphes:prop}
Soit $\cG=(R,V,E)$ une for\^et enracin\'ee, soit $K$ un corps,
soit $(x_\alpha)_{\alpha\in V}$ une famille d'ind\'etermin\'ees
sur $K$ index\'ee par $V$, et soit $\varphi\colon E\to K$ une
fonction.  Pour tout $\beta\in V$, on pose
\[
 L_\beta
 = \begin{cases}
   x_\beta
     &\text{si $\beta\in R$,}\\
   x_\beta-\varphi(\{\alpha,\beta\})x_\alpha
     &\text{si $\beta\in S_\cG(\alpha)$ avec $\alpha\in V$.}
   \end{cases}
\]
Alors, en \'etendant l'ordre partiel sur $V$ \`a un ordre
total, la matrice des formes lin\'eaires $(L_\beta)_{\beta\in V}$
dans la base $(x_\alpha)_{\alpha\in V}$ est triangulaire
inf\'erieure avec $1$ partout sur la diagonale.
\end{proposition}

%
%

\section{Chemins de croissance maximale}
\label{sec:descente}

Dans ce paragraphe, on fixe un polyn\^ome unitaire
non constant $f(z)\in\bC[z]$, un sous-ensemble convexe
compact $\cK$ de $\bC$ contenant toutes les racines de $f$,
et un disque ferm\'e $D$ de $\bC$ contenant $\cK$.
On note $N$ le degr\'e de $f$, et $R$ le rayon de $D$.
Le but principal de ce paragraphe est de d\'emontrer le
r\'esultat suivant.

\begin{theoreme}
\label{descente:thm}
Soit $\beta\in \cK$. Il existe une racine $\alpha$ de $f$ et un chemin
$\gamma\colon[0,1]\to\bC$ joignant $\gamma(0)=\alpha$ \`a $\gamma(1)=\beta$,
tel que $f(\gamma(t))=tf(\beta)$ pour tout $t\in[0,1]$.  L'image d'un tel chemin
est contenue dans $\cK$, de longueur au plus $\pi RN$.
\end{theoreme}

Par \emph{chemin}, on entend ici une fonction $\gamma\colon I\to\bC$
continue et diff\'eren\-tiable par morceaux sur un intervalle
ferm\'e $I$ de $\bR$. Pour un chemin $\gamma$ comme dans l'\'enonc\'e
du th\'eor\`eme, $\gamma(0)$ est forc\'ement une racine de
$f$ et on a $\max\{|f(\gamma(t))|\,;\,0\le t\le 1\}=|f(\beta)|$. En fait, on
va voir que $\gamma$ est un chemin de croissance maximale pour $|f|$.

Pour la preuve, on consid\`ere le polyn\^ome $f$ comme un rev\^etement
de surfaces de Riemann $f\colon\bC\to\bC$ de degr\'e $N$, ramifi\'e
en un nombre fini de points.  Alors tout chemin
$\gamma\colon[0,1]\to\bC$ se rel\`eve en $N$ chemins
$\gamma_1,\dots,\gamma_N\colon[0,1]\to\bC$ tels que
$f^{-1}(\gamma(t))=\{\gamma_1(t),\dots,\gamma_N(t)\}$
pour tout $t\in[0,1]$.  Ces derniers ne sont pas uniques en
g\'en\'eral, \`a cause de la ramificaton, et on les obtient
par recollement comme dans la preuve de
\cite[Theorem 4.14]{Fo1981}.  Pour un chemin $\gamma$
de la forme $\gamma(t)=tf(\beta)$
avec $f(\beta)\neq 0$, on en d\'eduit l'\'enonc\'e suivant.

\begin{lemme}
\label{descente:lemme:chemin}
Soit $\beta\in\bC$ avec $f(\beta)\neq 0$, et soit $m=m(\beta)\ge 0$
l'ordre de la d\'eriv\'ee de $f$ en $\beta$.
Alors, il existe $\delta\in(0,1)$ et $m+1$ chemins
$\gamma_0,\dots,\gamma_m$ de $[0,1]$ dans $\bC$ tels que
\begin{itemize}
 \item[(i)] $\gamma_0(1)=\cdots=\gamma_m(1)=\beta$,
 \item[(ii)] $f(\gamma_0(t))=\cdots=f(\gamma_m(t))=tf(\beta)$
   pour tout $t\in[0,1]$,
 \item[(iii)] $\gamma_0(t),\dots,\gamma_m(t)$ sont $m+1$
   nombres distincts pour chaque $t\in(1-\delta,1)$.
\end{itemize}
De plus, pour $j=0,1,\dots,m$ et pour chaque $t\in(0,1)$ tel que
$f'(\gamma_j(t))\neq 0$, la fonction $\gamma_j$ est
analytique en $t$ et sa d\'eriv\'ee $\gamma_j'(t)$ pointe dans
la direction o\`u la norme $|f|$ de $f$ cro\^{\i}t le
plus vite.
\end{lemme}

La derni\`ere affirmation du lemme signifie que $\gamma_0,\dots,
\gamma_m$ sont des chemins de croissance maximale pour la norme
de $f$.  En fait, c'est vrai pour tout chemin $\gamma$
tel que $f(\gamma(t))=ct$ ($0\le t\le 1$)
avec $c\in\bC\setminus\{0\}$ fix\'e car la courbe $t\mapsto ct$
d\'ecrit une demi-droite orthogonale aux cercles centr\'es \`a l'origine.
Comme l'application $f\colon\bC\to\bC$ est conforme en dehors
des points de ramification, la pr\'eimage $\gamma$ de cette
courbe est orthogonale aux courbes de niveau de $|f|$
en dehors de ces points. Nous revisiterons la construction
des $\gamma_j$ au lemme \ref{arch:lemme:regions}.

\subsection*{D\'emonstration du th\'eor\`eme \ref{descente:thm}}
Si $f(\beta)=0$, le chemin constant $\gamma(t)=\beta$ pour
tout $t\in [0,1]$ est le seul chemin possible et il poss\`ede
les popri\'et\'es requises.  Supposons donc que $f(\beta)\neq 0$.
Alors le lemme pr\'ec\'edant fournit un chemin
$\gamma\colon[0,1]\to\bC$ avec $\gamma(1)=\beta$ et
$f(\gamma(t))=tf(\beta)$ pour tout $t\in [0,1]$.  Fixons un
tel chemin.  Pour les calculs, on note
$\alpha_1,\dots,\alpha_s$ les racines distinctes de $f$
danc $\bC$ et $n_1,\dots,n_s$ leurs multiplicit\'es
respectives de sorte que
\begin{equation*}
 f(z)=(z-\alpha_1)^{n_1}\cdots(z-\alpha_s)^{n_s}.
 \label{descente:eq:f(z)}
\end{equation*}
On d\'esigne aussi par
$B$ l'ensemble des z\'eros de la d\'eriv\'ee $f'$ de $f$.

Le th\'eor\`eme de Gauss-Lucas nous apprend que $B$
est contenu dans l'enveloppe convexe des racines
de $f$, donc $B\subset \cK$.  Le fait que l'image de $\gamma$
soit contenue dans $\cK$ se d\'emontre de mani\`ere semblable.
En effet, supposons que ce ne soit pas le cas.  Alors,
comme $\cK$ est convexe, il existe un demi-plan qui contient
$\cK$ mais pas l'image de $\gamma$.  Autrement dit, il existe
$a,b\in\bC$ avec $|a|=1$ tels que $\Re(az+b)\le 0$ pour
tout $z\in \cK$ et que $\Re(a\gamma(t)+b)>0$ pour au moins un
$t\in [0,1]$.  Choisissons $t_0\in[0,1]$
tel que $\Re(a\gamma(t_0)+b)$ soit maximal, et
posons $z_0=\gamma(t_0)$.  Comme $\Re(az_0+b)>0$,
on a $z_0\notin\cK$, donc $t_0\in (0,1)$ et $z_0\notin B$.
Par suite $\gamma$ est diff\'erentiable en $t_0$ avec
$\Re(a\gamma'(t_0))=0$.  Par ailleurs, en d\'erivant
les deux membres de l'\'egalit\'e $f(\gamma(t))=tf(\beta)$
en $t=t_0$, on obtient
\[
 a\gamma'(t_0)
  = \frac{af(\beta)}{f'(\gamma(t_0))}
  = \frac{af(z_0)}{t_0f'(z_0)}
  = \left(\sum_{\ell=1}^s\frac{t_0 n_\ell}{a(z_0-\alpha_\ell)}\right)^{-1}.
\]
Comme $\Re(a(z_0-\alpha_\ell))=\Re(az_0+b)-\Re(a\alpha_\ell+b)
\ge \Re(az_0+b) >0$ pour $\ell=1,\dots,s$, on en d\'eduit que
$\Re(a\gamma'(t_0))>0$, une contradiction.

Pour estimer la longueur $L(\gamma)$ du chemin $\gamma$,
on utilise la formule de Cauchy-Crofton
\begin{align*}
 L(\gamma)
  =\frac{1}{4}
   \int_0^{2\pi}A(\theta)\,d\theta
 \quad
 &\text{o\`u}\quad
 A(\theta)=\int_{-\infty}^\infty N(r,\theta)\,dr\,,\\
 &\text{et}\quad
 N(r,\theta)
  =\card \{t\in[0,1]\,;\, \Re(\gamma(t)e^{-i\theta})=r\}
\end{align*}
(voir par exemple la jolie d\'emonstration de \cite{AD1997}).
Fixons $r,\theta\in\bR$ et consid\'erons le polyn\^ome
\[
 g_{r,\theta}(u)
  = \Im\left(\frac{f((r+iu)e^{i\theta})}{f(\beta)}\right)
  \in \bR[u].
\]
Si $t_0\in[0,1]$ satisfait $\Re(\gamma(t_0)e^{-i\theta})=r$,
on peut \'ecrire $\gamma(t_0)=(r+iu_0)e^{i\theta}$
avec $u_0\in\bR$.  On en d\'eduit $f((r+iu_0)e^{i\theta})=t_0f(\beta)$
et par suite $g_{r,\theta}(u_0)=0$.  Comme $\gamma$ est
injective sur $[0,1]$ (car $f\circ\gamma$ l'est), cela
signifie que $N(r,\theta)$ est au plus \'egal au nombre de racines
r\'eelles de $g_{r,\theta}$.  Or, comme $f$ est de
degr\'e $N$, le polyn\^ome $g_{r,\theta}(u)$ est de degr\'e
au plus $N$ et son coefficient de $u^N$ est
$\Im((ie^{i\theta})^N/f(\beta))$.  Donc, si on excepte
les $2N$ valeurs de $\theta\in[0,2\pi)$ pour lesquelles
ce coefficient est nul, on a n\'ecessairement
$g_{r,\theta}\neq 0$, donc $N(r,\theta)\le N$.

Pour $\theta$ fix\'e, l'ensemble $\{\Re(ze^{-i\theta})\,;\, z\in D\}$
est un intervalle $I_\theta$ de $\bR$ de longueur $2R$.   Comme l'image
de $\gamma$ est contenue dans $\cK\subset D$,
on a forc\'ement $N(r,\theta)=0$ si $r\notin I_\theta$. On en d\'eduit
que $A(\theta) \le 2R N$ sauf pour au plus $2N$ valeurs de
$\theta\in [0,2\pi)$, et par suite $L(\gamma)\le \pi R N$.

%
%

\section{Chemins entre racines complexes}
\label{sec:arch}

Comme au paragraphe pr\'ec\'edant, on fixe un polyn\^ome unitaire
non constant $f(z)\in\bC[z]$.  On note $N$ son degr\'e,
$A=\{\alpha_1,\dots,\alpha_s\}$ l'ensemble de ses racines
complexes, $\cK$ l'enveloppe convexe de $A$, et $R$ le
rayon d'un disque ferm\'e $D$ qui contienne $A$.
On note aussi $B=\{\beta_1,\dots,\beta_{p}\}$ l'ensemble
des racines de $f'(z)$ qui ne sont pas racines de $f(z)$,
c'est-\`a-dire l'ensemble des z\'eros de la d\'eriv\'ee
logarithmique $f'(z)/f(z)$.  Alors on peut \'ecrire
\begin{align}
 f(z)&=(z-\alpha_1)^{n_1}\cdots(z-\alpha_s)^{n_s},
 \label{arch:eq:f(t)}\\
 f'(z)&=N(z-\alpha_1)^{n_1-1}\cdots(z-\alpha_s)^{n_s-1}
        (z-\beta_1)^{m_1}\cdots(z-\beta_p)^{m_p},
 \label{arch:eq:f'(t)}
\end{align}
pour des entiers $n_1,\dots,n_s\ge 1$ de somme $N$,
et des entiers $m_1,\dots,m_p\ge 1$ de somme $s-1$.

Pour tout $\beta\in\bC$, on d\'esigne par $m(\beta)$ l'ordre
de $f'(z)$ en $\beta$.  Avec cette
notation, on a $m_j=m(\beta_j)$ pour $j=1,\dots,p$.
Le but de ce paragraphe est de montrer le r\'esultat suivant.

\begin{theoreme}
\label{arch:thm:graphe}
Il existe un arbre $\cG$ qui poss\`ede les propri\'et\'es suivantes:
\begin{itemize}
 \item[(i)] L'ensemble de ses sommets est $A$.
 \item[(ii)] Il poss\`ede $s-1$ ar\^etes, chacune index\'ee
      par un \'el\'ement de $B$.
 \item[(iii)] Pour chaque $\beta\in B$, il existe exactement
      $m(\beta)$ ar\^etes de $G$ index\'ees par $\beta$.
 \item[(iv)] Si $\{\alpha,\alpha'\}$ est une ar\^ete de $G$ index\'ee
      par $\beta$, il existe un chemin $\gamma\colon [0,1]\to \bC$
      de longueur au plus $2\pi R N$, contenu dans $\cK$, joignant
      $\gamma(0)=\alpha$ \`a $\gamma(1)=\alpha'$, tel que
     \[
      \gamma(1/2)=\beta
      \et
      \max_{0\le t\le 1} |f(\gamma(t)| = |f(\beta)|.
     \]
\end{itemize}
\end{theoreme}

Lorsque les racines de $f(z)$ sont toutes r\'eelles, on a
$f(z)\in\bR[z]$ et on peut donner une d\'emonstration tr\`es
simple du th\'eor\`eme.  Pour cela, on suppose d'abord
$\alpha_1<\cdots<\alpha_s$ list\'ees
par ordre croissant.  Alors, dans chaque intervalle
$[\alpha_j,\alpha_{j+1}]$ avec $1\le j\le s-1$, la fonction
$|f(z)|$ atteint son maximum en un z\'ero $\beta_j$ de $f'(z)$
avec $\alpha_j<\beta_j<\alpha_{j+1}$.  Comme $B$ est de
cardinalit\'e $p\le s-1$, cela \'epuise tous les
\'el\'ements de $B$: on a $p=s-1$ et $m_1=\cdots=m_{s-1}=1$.
On prend pour $\cG$ le graphe de sommets $A$, dont les ar\^etes
sont les paires $\{\alpha_j,\alpha_{j+1}\}$ index\'ees par
$\beta_j$ pour $j=1,\dots,s-1$.  Alors $\cG$ est un arbre et,
pour chaque $j=1,\dots,s-1$, le chemin $\gamma_j$ affine
par morceaux avec $\gamma_j(0)=\alpha_j$, $\gamma_j(1/2)=\beta_j$
et $\gamma_j(1)=\alpha_{j+1}$ poss\`ede la propri\'et\'e requise
en (iv); sa longueur est $\alpha_{j+1}-\alpha_j\le 2R$.

\subsection*{\'Etape 1}
La preuve du cas g\'en\'eral utilise plusieurs lemmes.
Pour tout $\beta\in B$, on choisit une fois pour toute
$m(\beta)+1$ chemins $\gamma_{\beta,0},\dots,\gamma_{\beta,m(\beta)}$
d'extr\'emit\'e $\beta$ comme au lemme \ref{descente:lemme:chemin}.
Alors on a $\gamma_{\beta,j}(0)\in A$ pour $j=0,\dots,m(\beta)$.
Notre objectif est de montrer que ces $m(\beta)+1$ points
de $A$ sont distincts et que le graphe $\cG$ de sommets
$\alpha_1,\dots,\alpha_s$ et d'ar\^etes
$\{\gamma_{\beta,0}(0),\gamma_{\beta,j}(0)\}$
avec $\beta\in B$ et $1\le j \le m(\beta)$ poss\`ede
les propri\'et\'es (i) \`a (iv) du th\'eor\`eme.  On commence par
\'etablir la propri\'et\'e (iv).

\begin{lemme}
\label{arch:lemme:longueur}
Soient $\beta\in B$ et $j\in\{1,\dots,m(\beta)\}$.  Alors
le chemin $\tgamma$ de $\gamma_{\beta,0}(0)$ \`a
$\gamma_{\beta,j}(0)$ donn\'e par
\[
 \tgamma(t)
  =\begin{cases}
   \gamma_{\beta,0}(2t) &\text{si $0\le t\le 1/2$,}\\
   \gamma_{\beta,j}(2-2t) &\text{si $1/2\le t\le 1$,}
   \end{cases}
\]
est contenu dans $\cK$, de longueur au
plus $2\pi R N$.  De plus, il satisfait
\[
 \tgamma(1/2)=\beta
 \et
 \max_{0\le t\le 1} |f(\tgamma(t)| = |f(\beta)|.
\]
\end{lemme}

\begin{proof}[D\'emonstration]
On a $B\subset \cK$ en vertu du th\'eor\`eme de Gauss-Lucas.
Alors, pour tout $\beta\in B$, le th\'eor\`eme \ref{descente:thm}
montre que les chemins $\gamma_{\beta,0}$ et $\gamma_{\beta,j}$
sont contenus dans $\cK$ de longueur
au plus $\pi RN$.  La conclusion suit car se sont
des chemins de croissance maximale pour $|f|$.
\end{proof}

\subsection*{\'Etape 2}
On d\'emontre d'abord le r\'esultat ci-dessous o\`u $S=\bC\cup\{\infty\}$
repr\'esente la sph\`ere de Riemann munie de sa topologie usuelle.
On l'utilise ensuite pour construire un arbre $\cH$ sur $A\cup B$.

\begin{lemme}
\label{arch:lemme:regions}
Soit $\beta\in B$ et soit $m=m(\beta)$.  Il existe $\delta>0$ et
$m+1$ fonctions continues $\gamma^+_0,\dots,\gamma^+_m$ de
$[1,\infty]$ dans $S=\bC\cup\{\infty\}$
telles que
\begin{itemize}
 \item[(i)] $\gamma^+_0(1)=\cdots=\gamma^+_m(1)=\beta$,
 \item[(ii)] $f(\gamma^+_0(t))=\cdots=f(\gamma^+_m(t))=tf(\beta)$
   pour tout $t\in[1,\infty]$,
 \item[(iii)] $\gamma^+_0(t),\dots,\gamma^+_m(t)$ sont $m+1$
   nombres distincts pour chaque $t\in(1,1+\delta)$.
\end{itemize}
Alors, les courbes $\Gamma^+_0=\gamma^+_0([1,\infty]),\dots,
\Gamma^+_m=\gamma^+_m([1,\infty])$ 
ne se rencontrent qu'aux points $\beta$ et $\infty$
sur $S$.  De plus, leur compl\'ement
$S\setminus(\Gamma^+_0\cup\cdots\cup\Gamma^+_m)$
est la r\'eunion de $m+1$ ouverts connexes disjoints
$\cR_0,\dots,\cR_m$ de $\bC$ tels que $\gamma_{\beta,j}([0,1))\subseteq
\cR_j$ pour $j=0,\dots,m$.
\end{lemme}

La preuve est illustr\'ee par la figure \ref{fig1}. Elle utilise le th\'eor\`eme
de Jordan sur les courbes ferm\'ees simples.

\begin{figure}[ht]
 \begin{tikzpicture}
    \clip(-5.5,-3) rectangle (5.5,3);
    \draw[fill] (0, 2) circle (0.05);
    \node[above] at (0,2.2) {$\infty$};
    \draw[->-] (0,-2) .. controls (5.7,-5) and (5.7,5) .. (0,2); 
    \draw[->-] (0,2) .. controls (4,3) and (4,-3) .. (0,-2); 
    \draw[->-] (0,-2) .. controls (2,-1) and (2,1) .. (0,2);
    \draw[->-] (0,2) .. controls (0.5,1) and (0.5,-1) ..  (0,-2);
    \draw[->-] (0,-2) .. controls (-5,-5) and (-5,5) .. (0,2);
    \draw[->-] (0,2) .. controls (-3.5,3) and (-3.5,-3) .. (0,-2);
    \draw[->-] (0,-2) .. controls (-2,-1) and (-2,1) .. (0,2);
    \node[left] at (-4.7,0) {$\cdots$};
    \node[left, above] at (-4.1,-0.2) {$\Gamma^+_1$};
    \node[left, above] at (-3,-0.2) {$\Gamma^-_0$};
    \node[left, above] at (-1.8,-0.2) {$\Gamma^+_0$};
    \node[left, above] at (0,-0.2) {$\Gamma^-_m$};
    \node[left, above] at (1.1,-0.2) {$\Gamma^+_m$};
    \node[left, above] at (2.4,-0.2) {$\Gamma^-_{m-1}$}; 
    \node[left, above] at (3.7,-0.2) {$\Gamma^+_{m-1}$}; 
    \node[right] at (4.5,0) {$\cdots$};
    \draw[fill] (0,-2) circle (0.05);
    \node[below] at (0,-2.2) {$\beta$};
    \node[left] at (0.2,-1.3) {$\cR_m$};
    \node[left] at (3,-2.2) {$\cR_{m-1}$}; 
    \node[left] at (-1.3,-2.2) {$\cR_{0}$};
    \draw[fill] (-2.27, -1.3) circle (0.05);
    \node[left] at (-2.27,-1.3) {$\gamma^-_0(0)$};
    \draw[fill] (0.34, -0.7) circle (0.05);
    \node[left] at (0.34,-0.7) {$\gamma^-_{m}(0)$};
    \draw[fill] (2.6, -1.3) circle (0.05); 
    \node[left] at (2.6,-1.3) {$\gamma^-_{m-1}(0)$}; 
 \end{tikzpicture}
\caption{Illustration for the proof of Lemma \ref{arch:lemme:regions}.}
\label{fig1}
\end{figure}
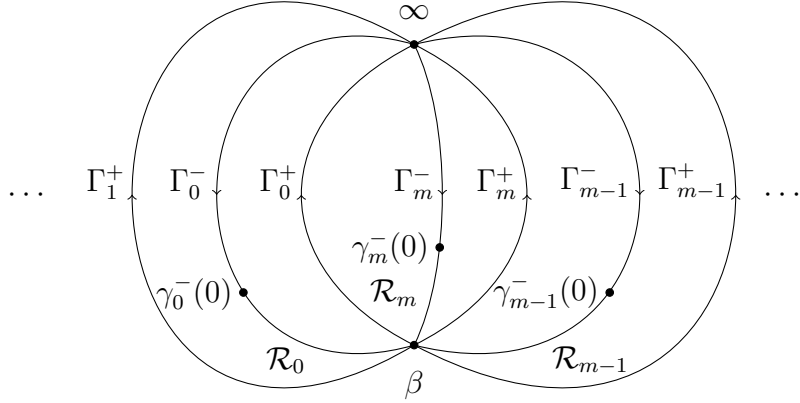

\begin{proof}[D\'emonstration]
En posant $\ell=m+1$, on obtient $f(z)=f(\beta)(1+(z-\beta)^\ell g(z))$
o\`u $g(z)$ est un polyn\^ome avec $g(\beta)\neq 0$.  Alors pour
$\epsilon>0$ assez petit, il existe un voisinage ouvert $V$ de $\beta$
et une fonction biholomorphe $h$ de $V$ dans
$B(0,\epsilon)=\{z\in\bC\,;\,|z|<\epsilon\}$ satisfaisant
$h(\beta)=0$ et
\[
 f(z)=f(\beta)(1+h(z)^\ell)
\]
pour tout $z\in V$.  Fixons un tel choix de $\epsilon$, $V$ et $h$,
puis posons $\delta=\epsilon^\ell$ et $\rho=e^{\pi i/\ell}$.
Pour $j=0,\dots,m$, on d\'efinit une fonction continue
$\gamma^+_j\colon [1,1+\delta)\to V$ en posant
\[
  \gamma^+_j(t)=h^{-1}\left(\rho^{2j}(t-1)^{1/\ell}\right)
  \quad
  (1\le t<1+\delta).
\]
Alors, pour $t\in(1,1+\delta)$ fix\'e, les nombres
$z=\gamma^+_0(t),\dots,\gamma^+_m(t)$ sont les $\ell$ solutions
distinctes de $f(z)=tf(\beta)$ avec $z\in V$.
En particulier, $\gamma^+_0,\dots,\gamma^+_m$ satisfont
les conditions (i) et (iii) du lemme, ainsi que (ii) pour
tout $t\in[1,1+\delta)$.  Pour chaque $j=0,\dots,m$, on choisit
un prolongement de $\gamma^+_j$ en une fonction continue
$\gamma^+_j\colon[1,\infty]\to S$ satisfaisant
$f(\gamma^+_j(t))=tf(\beta)$ pour tout $t\in[1,\infty]$.

De m\^eme, pour $j=0,\dots,m$, on d\'efinit une fonction continue
$\gamma^-_j\colon (1-\delta,1]\to V$ en posant
\[
\gamma^-_j(t)=h^{-1}\left(\rho^{2j+1}(1-t)^{1/\ell}\right)
  \quad (1-\delta< t\le 1).
\]
Pour $t\in(1-\delta,1)$ fix\'e, les nombres
$z=\gamma^-_0(t),\dots,\gamma^-_m(t)$ sont les $\ell$ solutions
distinctes de $f(z)=tf(\beta)$ avec $z\in V$, donc ils constituent
une permutation de $\gamma_{\beta,0}(t),\dots,\gamma_{\beta,m}(t)$.
Cette permutation \'etant ind\'ependante de $t$, il n'y a pas
de perte de g\'en\'eralit\'e \`a supposer que $\gamma^-_j$ est
la restriction de $\gamma_{\beta,j}$ \`a $(1-\delta,1]$ pour
$j=0,\dots,m$.  On \'etend alors
chaque $\gamma_{\beta,j}\colon [0,1]\to\bC$ en
une fonction continue $\gamma^-_j\colon[-\infty,1]\to S$
telle que $f(\gamma^-_j(t))=tf(\beta)$ pour tout $t\in[-\infty,1]$.

Posons $\Gamma^-_j = \gamma^-_j([-\infty,1])$ et
$\Gamma^+_j = \gamma^+_j([1,\infty])$ pour $j=0,\dots,m$,
et fixons temporairement $j,k\in\{0,1,\dots,m\}$. Les courbes
$\Gamma^-_j$ et $\Gamma^+_k$ ne se rencontrent qu'aux points
$\beta$ et $\infty$ car si $\gamma^-_j(t)=\gamma^+_k(u)$ avec
$t\in[-\infty,1]$ et $u\in[1,\infty]$, alors
$tf(\beta) = uf(\beta)$, donc $t=u=1$ ou $-t=u=\infty$.
Supposons maintenant $j<k$.  Comme les courbes $\Gamma^+_j$ et
$\Gamma^+_k$ se rencontrent \`a l'infini, il existe un plus
petit $r\in[1+\delta,\infty]$ tel que $\gamma^+_j(r)=\gamma^+_k(r)$.
Pour ce choix de $r$, la r\'eunion
$\gamma^+_j([1,r]) \cup \gamma^+_k([1,r])$ est une courbe
ferm\'ee simple $\Gamma$.  En vertu du th\'eor\`eme de
Jordan, son compl\'ement dans $S$ est donc la r\'eunion
de deux ouverts connexes $\cR$ et $\cR'$ de fronti\`ere
$\Gamma$.  Par ailleurs, on a
\[
 V\cap\Gamma
  = \gamma^+_j([1,1+\delta))\cup\gamma^+_k([1,1+\delta))
  = h(P)
 \quad
 \text{o\`u}
 \quad
 P=[0,\epsilon)\rho^{2j}\cup[0,\epsilon)\rho^{2k}.
\]
De plus, $B(0,\epsilon)\setminus P$ est la r\'eunion de deux ouverts
connexes disjoints $\cU$ et $\cU'$ (des secteurs du disque
$B(0,\epsilon)$),  o\`u $\cU$ contient les rayons
$(0,\epsilon)\rho^{2i+1}$ avec $j\le i<k$ et o\`u $\cU'$
contient ceux avec $0\le i<j$ ou $k\le i\le m$.
Comme $h\colon V\to B(0,\epsilon)$ est un hom\'eomorphisme,
$h(\cU)$ et $h(\cU')$ sont des ouverts connexes disjoints
de $S$ dont l'union est $V\setminus\Gamma$.  On peut supposer
que $h(\cU)\subset \cR$ et $h(\cU')\subset \cR'$.  Alors,
on obtient
\[
 \gamma^-_i((1-\delta,1))=h((0,\epsilon)\rho^{2i+1})
 \subseteq
 \begin{cases}
   \cR  &\text{si $j\le i<k$,}\\
   \cR' &\text{sinon.}
 \end{cases}
\]
Or, $\cR$ et $\cR'$ partagent la m\^eme fronti\`ere, contenue dans
$\Gamma^+_j \cup \Gamma^+_k$.  Donc aucun des ensembles
$\Gamma^-_i \setminus \{\beta,\infty\}=\gamma_i^-((-\infty,1))$
ne rencontre cette fronti\`ere.  Comme ce sont des courbes connexes,
on en d\'eduit que $\Gamma^-_i \setminus \{\beta,\infty\}$ est
contenu dans $\cR$ si $j\le i<k$ et dans $\cR'$ sinon.  En particulier,
aucun des ouverts $\cR$ ou $\cR'$ n'est born\'e et par suite,
on doit avoir $r=\infty$.  Cela signifie que $\Gamma^+_j$
et $\Gamma^+_k$ ne se rencontrent qu'en $\beta$ et en $\infty$.

Avec les notations ci-dessus, on d\'efinit $\cR_j=\cR$ pour le choix
de $j\in\{0,\dots,m-1\}$ et de $k=j+1$.  On d\'efinit aussi
$\cR_m=\cR'$ pour le choix de $j=0$ et de $k=m$.  Ce sont des
ouverts connexes de $\bC$ avec $\gamma_{\beta,j}([0,1))
\subset \Gamma^-_j \setminus \{\beta,\infty\} \subset \cR_j$ pour
$j=0,\dots,m$.  Il reste \`a montrer que $\cR_0,\dots,\cR_m$ sont
disjoints deux \`a deux.  On note d'abord que si $j\neq k$,
alors $\cR_j \not\subseteq \cR_k$ car
$\Gamma^-_j\setminus\{\beta,\infty\}$ est contenu dans $\cR_j$
mais pas dans $\cR_k$.  Donc si $\cR_j$ et $\cR_k$ ont au moins
un point commun, alors $\cR_j$ rencontre la fronti\`ere
de $\cR_k$.  Par suite $\cR_j$ contient au moins un point
de $\Gamma^+_i\setminus\{\beta,\infty\}$
pour un indice $i\in\{0,1,\dots,m\}$.  Par contre, le choix
de $\cR_j$ implique que $\gamma^+_i(t)\notin\cR_j$ pour
tout $t\in(1,1+\delta)$.  Donc la courbe
$\Gamma^+_i\setminus\{\beta,\infty\}$ n'est pas contenue
dans $\cR_j$ et, comme elle est connexe, elle rencontre
la fronti\`ere de $\cR_j$ sans \^etre contenue dedans.  C'est
impossible car cette fronti\`ere est l'union de deux
courbes parmi $\Gamma^+_0,\dots,\Gamma^+_m$.
\end{proof}

\begin{lemme}
\label{arch:graphe:H}
Pour tout $\beta\in B$, les $m(\beta)+1$ points
$\gamma_{\beta,j}(0)\in A$ avec $0\le j\le m(\beta)$ sont
distincts deux-\`a-deux.  De plus, soit $\cH$ le graphe
dont l'ensemble des sommets est $A\cup B$
et dont les ar\^etes sont les couples $\{\beta,\gamma_{\beta,j}(0)\}$
avec $\beta\in B$ et $0\le j \le m(\beta)$.  Alors $\cH$
est un arbre.
\end{lemme}

\begin{proof}[D\'emonstration]
La premi\`ere affirmation d\'ecoule directement du lemme
pr\'ec\'edant car, pour $\beta\in B$ et $m=m(\beta)$, ce lemme
fournit des ouverts connexes disjoints $\cR_0,\dots,\cR_m$
tels que $\gamma_{\beta,j}(0)\in \cR_j$ pour $j=0,\dots,m$.

Supposons que $\cH$ ne soit pas une for\^et.  Alors $\cH$
contient un cycle simple, c'est-\`a-dire une cha\^{\i}ne
\'el\'ementaire $(a_1,\dots,a_k)$ avec $k\ge 3$ telle que
$\{a_k,a_1\}$ soit une ar\^ete de $\cH$.  Alors, $k$ est un
entier pair et les $a_i$ appartiennent \`a $A$ ou \`a $B$
en alternance selon la parit\'e de $i$.  Quitte \`a permuter
cycliquement les \'el\'ements de cette cha\^{\i}ne, on peut
supposer que $a_1\in B$ et que $|f(a_1)|\ge |f(a_i)|$ pour
$i=1,\dots,k$. Soit $m=m(a_1)$ et soient $\cR_0,\dots,\cR_m$
les ouverts connexes associ\'es au point $a_1\in B$, comme au
lemme \ref{arch:lemme:regions}.  Pour tout point $z\neq a_1$ en
dehors de ces ouverts, on a $f(z)=tf(a_1)$ pour un
nombre r\'eel $t>1$, donc $|f(z)|>|f(a_1)|$.  On pose $a_{k+1}=a_1$
et, pour $i=1,\dots,k$, on note $\gamma_i$ le chemin de la forme
$\gamma_{\beta,j}$ qui lie $a_i$ et $a_{i+1}$. Pour tout $t\in[0,1]$,
on a $f(\gamma_i(t))=t f(a_i)$ si $i$ est impair et
$f(\gamma_i(t)) = t f(a_{i+1})$ si $i$ est pair.  Donc, dans les
deux cas, on trouve $|f(\gamma_i(t))| \le |f(a_1)|$, avec
l'in\'egalit\'e stricte si $t\neq 1$.
Comme $a_1,\dots,a_k$ sont distincts et que $\gamma_i(1) \in
\{a_3,\dots,a_{k-1}\}$ si $2\le i\le k-1$, on en d\'eduit que
la courbe
\[
 \Gamma=\gamma_1([0,1))\cup
        \gamma_2([0,1])\cup\cdots\cup\gamma_{k-1}([0,1])
        \cup\gamma_k([0,1))
\]
est contenue dans $\cR_0\cup\cdots\cup\cR_m$.  Comme c'est
un sous-ensemble connexe de $\bC$, elle est donc toute enti\`ere
contenue dans $\cR_j$ pour un entier $j$.  Comme $\gamma_1(1)=
\gamma_k(1)=a_1$, cela implique que $\gamma_1=\gamma_k$, donc
$a_2=\gamma_1(0)=\gamma_k(0)=a_k$, ce qui est impossible.

Ainsi $\cH$ est une for\^et.  Alors, le nombre de ces composantes
connexes est \'egal au nombre de ses sommets moins celui de ses ar\^etes,
c'est-\`a-dire
\[
 |A\cup B| - \sum_{\beta\in B} (m(\beta)+1)
  = s - \sum_{\beta\in B} m(\beta) =1.
\]
Donc $\cH$ est un arbre.
\end{proof}

\subsection*{\'Etape 4. Preuve du th\'eor\`eme \ref{arch:thm:graphe}}
Soit $\cG$ le graphe dont l'ensemble des sommets est $A$ et dont
les ar\^etes sont les paires
\begin{equation}
 \label{etape5:eq}
 \{\gamma_{\beta,0}(0),\gamma_{\beta,j}(0)\}
 \quad
 \big(\beta\in B, \ 1\le j\le m(\beta)\big).
\end{equation}
Comme $\cH$ est connexe, le graphe $\cG$ l'est aussi. Comme
$\cG$ poss\`ede $s=|A|$ sommets et que $\sum_{\beta\in B} m(\beta)=s-1$,
on en d\'eduit que les $s-1$ ar\^etes \eqref{etape5:eq} sont
distinctes et que $\cG$ est un arbre. En particulier, pour chaque
$\beta\in B$, le graphe $\cG$ poss\`ede $m(\beta)$ ar\^etes index\'ees
par $\beta$ et le lemme \ref{arch:lemme:longueur} montre que, pour chacune
d'elles, il existe un chemin qui remplit la condition (iv) du
th\'eor\`eme.

%
%

\section{Un calcul de semi-r\'esultant}
\label{sec:semi-resultant}

On montre d'abord la formule suivante.

\begin{proposition}
\label{semi-resultant:prop}
Avec les notations du paragraphe pr\'ec\'edant, on a
\[
 N^N \prod_{j=1}^p f(\beta_j)^{m_j}
  = \prod_{i=1}^s
      \Big( n_i^{n_i}\prod_{k\neq i}(\alpha_i-\alpha_k)^{n_k}\Big).
\]
\end{proposition}

Le membre de gauche de cette \'egalit\'e est le semi-r\'esultant de
$f(z)$ et de $f'(z)$ au sens de Chudnovski \cite{Br1977,Ch1984}.

\begin{proof}[D\'emonstration]
La formule pour la d\'eriv\'ee d'un produit appliqu\'ee \`a la
factorisation \eqref{arch:eq:f(t)} de $f(z)$ livre
\[
 f'(z)=(z-\alpha_1)^{n_1-1}\cdots(z-\alpha_s)^{n_s-1}g(z)
\]
o\`u
\[
 g(z)=\sum_{k=1}^s
      n_k(z-\alpha_1)\cdots\widehat{(z-\alpha_k)}\cdots(z-\alpha_s).
\]
En comparant avec la factorisation \eqref{arch:eq:f'(t)} de
$f'(z)$, on trouve aussi que
\[
 g(z)=N(z-\beta_1)^{m_1}\cdots(z-\beta_p)^{m_p}.
\]
En \'evaluant ces deux expressions pour $g(z)$ en $z=\alpha_k$, on
en d\'eduit la relation
\[
 N\prod_{j=1}^p(\alpha_k-\beta_j)^{m_j}
  = n_k\prod_{i\neq k} (\alpha_k-\alpha_i)
 \quad
 (1\le k\le s).
\]
Comme $m_1+\cdots+m_p=s-1$, ces \'egalit\'es se r\'e\'ecrivent
\[
 N\prod_{j=1}^p(\beta_j-\alpha_k)^{m_j}
  = n_k\prod_{i\neq k} (\alpha_i-\alpha_k)
 \quad
 (1\le k\le s).
\]
On en d\'eduit, comme annonc\'e,
\begin{align*}
 N^N \prod_{j=1}^p f(\beta_j)^{m_j}
   &= N^N \prod_{j=1}^p
       \Big( \prod_{k=1}^s(\beta_j-\alpha_k)^{n_k} \Big)^{m_j}\\
  &= \prod_{k=1}^s
       \Big( N \prod_{j=1}^p(\beta_j-\alpha_k)^{m_j} \Big)^{n_k}\\
  &= \prod_{k=1}^s
       \Big( n_k \prod_{i\neq k}(\alpha_i-\alpha_k) \Big)^{n_k}
   = \prod_{i=1}^s
      \Big( n_i^{n_i}\prod_{k\neq i}(\alpha_i-\alpha_k)^{n_k}\Big).
\qedhere
\end{align*}
\end{proof}

\begin{cor}
\label{semi-resultant:cor}
Avec les m\^emes notations, on a
\[
 N! \prod_{j=1}^p \big|f(\beta_j)\big|^{m_j}
  \le \prod_{i=1}^s
      \Big( n_i! \prod_{k\neq i}|\alpha_i-\alpha_k|^{n_k}\Big).
\]
\end{cor}

\begin{proof}[D\'emonstration]
Comme $N=n_1+\cdots+n_s$, on trouve
\[
 \frac{N!}{n_1!\cdots n_s!}
   \Big(\frac{n_1}{N}\Big)^{n_1}\cdots\Big(\frac{n_s}{N}\Big)^{n_s}
 \le \Big(\frac{n_1}{N}+\cdots+\frac{n_s}{N}\Big)^N
  = 1.
\]
On en tire $N!\prod_{i=1}^s n_i^{n_i} \le N^N \prod_{i=1}^s n_i!$\,,
et la conclusion suit.
\end{proof}

%
%

\section{Volume des composantes archim\'ediennes}
\label{sec:volarch}

On peut maintenant d\'emontrer la majoration de volume
du th\'eor\`eme \ref{res:principal} (i).  Les notations
sont celles du paragraphe \ref{sec:resultat}.

\begin{theoreme}
 \label{volarch:thm}
Soit $v$ une place archim\'edienne de $K$ et soit $\cC_{\un,v}$ le
corps convexe de $K_v^s$ d\'efini au paragraphe \ref{subsec:res}
pour le choix d'un $s$-uplet $\un=(n_1,\dots,n_s)\in\bN_+^s$.
Alors, on a
\[
 \mu_v(\cC_{\un,v})^{1/d_v}\le c_v N^{2s-2} |\Delta_\un|_v
 \quad
 \text{avec}
 \quad
 c_v = 2^s e^{sR_v} (2\pi R_v^s)^{s-1} |\Delta_\uun|_v^{-1}
\]
o\`u $N=n_1+\cdots+n_s$,
$R_v=\max_{1\le i<j\le s}|\alpha_i-\alpha_j|_v$,
et $\uun=(1,\dots,1)$.
\end{theoreme}

\begin{proof}[D\'emonstration]
Pour simplifier, on peut supposer $K$ plong\'e dans $\bC$ de sorte
que $|\alpha|_v=|\alpha|$ pour tout $\alpha\in K$. Quitte \`a permuter
$\alpha_1,\dots,\alpha_s$ si n\'ecessaire, on peut aussi supposer que
$n_1\ge\cdots\ge n_s$ forment une suite d\'ecroissante.  On note $D$
le disque ferm\'e de rayon $R_v$ et de centre
$(\alpha_1+\dots+\alpha_s)/s$ dans $\bC$.  Comme ce disque contient
$\alpha_1,\dots,\alpha_s$, il contient aussi l'enveloppe convexe $\cK$
de ces points.

Supposons d'abord que $n_1\ge 2$ et notons $r$ le plus grand indice
tel que $n_r\ge 2$.  On forme le polyn\^ome
\[
 f(z)=\frac{f_\un(z)}{(z-\alpha_1)\cdots(z-\alpha_s)}
     =\prod_{i=1}^r (z-\alpha_j)^{n_i-1}.
\]
L'ensemble de ses racines est $A=\{\alpha_1,\dots,\alpha_r\}$
et son degr\'e est $N-s$.  Sa d\'eriv\'ee s'\'ecrit
\[
 f'(z)=(N-s)(z-\alpha_1)^{n_1-2}\cdots(z-\alpha_r)^{n_r-2}
       (z-\beta_1)^{m_1}\cdots(z-\beta_p)^{m_p}
\]
o\`u $B=\{\beta_1,\dots,\beta_p\}$ est l'ensemble des racines
de $f'(z)$ en dehors de $A$, et o\`u $m_j$ est la multiplicit\'e
de $\beta_j$ pour $j=1,\dots,p$.  On choisit un arbre $G$
comme au th\'eor\`eme \ref{arch:thm:graphe} pour ce polyn\^ome
$f(z)$.  Par construction, l'ensemble de ses sommets est $A$.
On \'etend ensuite $G$ en un graphe $\tG$ sur
$\{\alpha_1,\dots,\alpha_s\}$ de la mani\`ere suivante.
Pour chaque $j=r+1,\dots,s$, on choisit
un chemin $\gamma_j\colon[0,1]\to\bC$ tel que $\gamma_j(1)=\alpha_j$
et $f(\gamma_j(t))=tf(\alpha_j)$ comme au th\'eor\`eme
\ref{descente:thm}.  Alors $\gamma_j(0)$ est une racine de $f$,
donc un \'el\'ement de $A$,
et on ajoute l'ar\^ete $\{\gamma_j(0), \alpha_j\}$ au graphe $G$.
Enfin, on enracine l'arbre $\tG$ ainsi construit en lui
choisissant $\alpha_1\in A$ comme racine.  Alors,
$\cC_{\un,v}$ est contenu dans l'ensemble $\cK_v$ des points
$(x_1,\dots,x_s)$ de $K_v^s$ satisfaisant
\[
 |x_1|_v
    \le e^{R_v} (N-1)!
\]
et
\[
 |x_ie^{\alpha_j-\alpha_i}-x_j|_v
 \le b_{i,j}
  := \max_{1\le k\le s}
      \left|
       \int_{\alpha_i}^{\alpha_j}
          f_{\un-\ue_k}(z)e^{\alpha_j-z}\dz
      \right|
\]
pour chaque arc $(\alpha_i,\alpha_j)$ de $\tG$ avec
$\alpha_i<\alpha_j$.  Comme $\tG$ est
un arbre enracin\'e, la proposition \ref{graphes:prop} montre que les $s$
formes lin\'eaires qui d\'efinissent $\cK_v$ sont lin\'eairement
ind\'ependantes, de d\'eterminant $\pm 1$.
Donc $\cK_v$ est un corps convexe de $K_v^s$ avec
\begin{equation}
 \label{volarch:eq1}
 \mu_v(\cC_{\un,v})^{1/d_v}
   \le \mu_v(\cK_v)^{1/d_v}
   \le 2^s e^{R_v} (N-1)!
       \prod_{(\alpha_i,\alpha_j)\in E} b_{i,j}
\end{equation}
o\`u $E$ d\'esigne l'ensemble des arcs (ar\^etes ordonn\'ees) de $\tG$.

Fixons temporairement $(\alpha_i,\alpha_j)\in E$ et
$k\in\{1,\dots,s\}$.  On a forc\'ement $i\le r$, c'est-\`a-dire
$\alpha_i\in A$.  Si $j\le r$, on a aussi $\alpha_j\in A$ et
$\{\alpha_i,\alpha_j\}$ est une ar\^ete de $G$.  Alors,
le th\'eor\`eme \ref{arch:thm:graphe} lui associe
un \'el\'ement $\beta$ de $B$ et
un chemin $\gamma\colon [0,1]\to\bC$ de longueur au plus
$2\pi R_v N$, contenu dans $\cK$, joignant $\alpha_i$ et $\alpha_j$,
tel que
\[
 \max_{0\le t\le 1}|f(\gamma(t))| = |f(\beta)|.
\]
On en d\'eduit
\begin{align*}
 \left|
       \int_{\alpha_i}^{\alpha_j}
          f_{\un-\ue_k}(z)e^{\alpha_j-z}\dz
 \right|
 &=
 \left|
       \int_{\alpha_i}^{\alpha_j}
          f(z)(z-\alpha_1)\cdots\widehat{(z-\alpha_k)}\cdots(z-\alpha_s)
          e^{\alpha_j-z} \dz
 \right|\\
 &\le 2\pi R_v N |f(\beta)|\,
      \max_{z\in \cK}
        \big|
        (z-\alpha_1)\cdots\widehat{(z-\alpha_k)}\cdots(z-\alpha_s)
          e^{\alpha_j-z}
        \big| \\
 &\le 2\pi R_v^s e^{R_v} N |f(\beta)|\,,
\end{align*}
car $|z-\alpha_\ell|\le R_v$ pour tout $z\in\cK$ et tout $\ell=1,\dots,s$.
Enfin, si $j>r$, on a $\alpha_i=\gamma_j(0)$ pour le chemin $\gamma_j$
choisi pr\'ec\'edemment.  En vertu du th\'eor\`eme \ref{descente:thm},
l'image de $\gamma_j$ est contenue dans $\cK$ de longueur au plus
$\pi R_v N\le 2\pi R_v N$.  Alors le m\^eme calcul que ci-dessus donne
\[
 \left|
       \int_{\alpha_i}^{\alpha_j}
          f_{\un-\ue_k}(z)e^{\alpha_j-z}\dz
 \right|
 \le 2\pi R_v^s e^{R_v} N |f(\alpha_j)|\,.
\]

Comme chaque $\beta_j$ est associ\'e \`a $m_j$ ar\^etes de $G$
et que $m_1+\cdots+m_p=r-1$, on d\'eduit de \eqref{volarch:eq1}
que
\begin{equation}
 \label{volarch:eq2}
 \mu_v(\cC_{\un,v})^{1/d_v}
   \le 2^s e^{R_v} (N-1)! \big(2\pi R_v^s e^{R_v} N)^{s-1}
       \prod_{j=1}^p |f(\beta_j)|^{m_j}
       \prod_{j=r+1}^s |f(\alpha_j)|\,.
\end{equation}
Comme $n_k=1$ pour $k>r$, le corollaire \ref{semi-resultant:cor} livre
\[
 (N-s)! \prod_{j=1}^p \big|f(\beta_j)\big|^{m_j}
  \le \prod_{i=1}^r
      \Big( (n_i-1)! \prod_{k\neq i}|\alpha_i-\alpha_k|^{n_k-1}\Big).
\]
Pour $i=r+1,\dots,s$, on trouve aussi
\[
 |f(\alpha_i)|=\prod_{k=1}^r|\alpha_i-\alpha_k|^{n_k-1}
   =(n_i-1)! \prod_{k\neq i} |\alpha_i-\alpha_k|^{n_k-1}.
\]
On en d\'eduit que
\[
 (N-s)! \prod_{j=1}^p |f(\beta_j)|^{m_j}
       \prod_{j=r+1}^s |f(\alpha_j)|
 \le \prod_{i=1}^s
      \Big( (n_i-1)! \prod_{k\neq i}|\alpha_i-\alpha_k|^{n_k-1}\Big)
  = \frac{|\Delta_\un|_v}{|\Delta_\uun|_v}.
\]
En substituant cette majoration dans \eqref{volarch:eq2},
on obtient $\mu_v(\cC_{\un,v})^{1/d_v}\le c_vN^{2s-2}|\Delta_\un|_v$,
comme annonc\'e.
\end{proof}

%
%

\section{Une for\^et aux places ultram\'etriques}
\label{sec:ultra}

Soit $v$ une place ultram\'etrique de $K$.  Dans ce paragraphe,
on utilise la terminologie du paragraphe \ref{sec:graphes} pour
b\^atir une for\^et sur un sous-ensemble quelconque, fini et
non vide, de $K_v$. On commence par une construction pr\'eliminaire.

\begin{proposition}
 \label{ultra:prop:arbre}
Soit $A$ un sous-ensemble fini non vide de $K_v$ et soit $\alpha_0\in A$.
Il existe un arbre $\cG$ enracin\'e en $\alpha_0$ avec $A$ pour ensemble
des sommets, tel que, pour tout choix de $\alpha,\beta,\gamma\in A$ avec
$\beta\in S_\cG(\alpha)$, on ait
\begin{equation}
 \label{ultra:prop:arbre:eq}
 \gamma\in D_\cG(\beta)
    \quad\Longleftrightarrow\quad
 |\alpha-\beta|_v >|\beta-\gamma|_v>0.
\end{equation}
\end{proposition}

\begin{proof}[D\'emonstration]
On proc\`ede par r\'ecurrence sur la cardinalit\'e $|A|$ de $A$.  Si $|A|=1$,
il n'y a rien \`a d\'emontrer. Supposons $|A|\ge 2$.  Soit $\rho$ la distance
maximale entre deux \'el\'ements de $A$, et soit $\{\alpha_0,\dots,\alpha_k\}$
un sous-ensemble maximal de $A$ contenant $\alpha_0$, dont les \'el\'ements
sont \`a distance mutuelle $|\alpha_i-\alpha_j|_v=\rho$ pour $0\le i<j\le k$.
Comme la distance est ultram\'etrique, on a $k\ge 1$ et les ensembles
\[
 A_i:=\{\beta\in A\,;\, |\alpha_i-\beta|_v<\rho\}
 \quad
 (0\le i \le k)
\]
forment une partition de $A$.  Pour $i=0,\dots,k$, on a $\alpha_i\in A_i$
et $|A_i|<|A|$, donc on peut supposer qu'il existe un arbre enracin\'e
$\cG_i=(\alpha_i,A_i,E_i)$ tel que la condition \eqref{ultra:prop:arbre:eq}
soit remplie, avec $\cG$ remplac\'e par $\cG_i$, pour tout choix de
$\alpha,\beta,\gamma \in A_i$ avec $\beta\in S_{\cG_i}(\alpha)$.  On pose
\[
 E=E_0\cup\cdots\cup E_k\cup\{\{\alpha_0,\alpha_1\},\dots,\{\alpha_0,\alpha_k\}\}.
\]
Alors $\cG=(\alpha_0,A,E)$ est un arbre enracin\'e en $\alpha_0$.  Soient
$\alpha,\beta,\gamma\in A$ avec $\beta\in S_\cG(\alpha)$, et soit $i$
l'indice pour lequel $\alpha\in A_i$.  Si $\beta\in A_i$, alors
$\beta\in S_{\cG_i}(\alpha)$ et $D_\cG(\beta)=D_{\cG_i}(\beta)$, donc
\[
 \gamma\in D_\cG(\beta)
    \quad\Longleftrightarrow\quad
 \gamma\in D_{\cG_i}(\beta)
    \quad\Longleftrightarrow\quad
 |\alpha-\beta|_v >|\beta-\gamma|_v>0.
\]
Si au contraire $\beta\in A_j$ avec $j\neq i$, on a forc\'ement $i=0$,
$\alpha=\alpha_0$ et $\beta=\alpha_j$.  Alors $|\alpha-\beta|_v=\rho$ et
$D_\cG(\beta)=A_j\setminus\{\alpha_j\}$, et on trouve encore
\[
 \gamma\in D_\cG(\beta)
    \quad\Longleftrightarrow\quad
 \rho>|\alpha_j-\gamma|_v>0
    \quad\Longleftrightarrow\quad
 |\alpha-\beta|_v >|\beta-\gamma|_v>0.
\]
Donc $\cG$ poss\`ede la propri\'et\'e requise.
\end{proof}

Comme la preuve le fait voir, le graphe $\cG$ ainsi construit
n'est en g\'en\'eral pas unique (car le choix de
$\alpha_1,\dots,\alpha_k\in A$ ne l'est pas).  On en d\'eduit la
construction suivante qui n'est g\'en\'eralement pas unique
non plus.

\begin{theoreme}
 \label{ultra:thm:foret}
Soit $A$ un sous-ensemble fini non vide de $K_v$, soit $\delta>0$,
et soit $R$ un sous-ensemble maximal de $A$ dont les \'el\'ements
sont \`a distance mutuelle sup\'erieure ou \'egale \`a $\delta$.
Alors, il existe une for\^et enracin\'ee $\cG$ dont $A$ est
l'ensemble des sommets et $R$ l'ensemble des racines, qui
poss\`ede les propri\'et\'es suivantes:
\begin{itemize}
\item[(i)] pour tout $\beta\in R$ et tout $\gamma\in A$, on a
\[
 \gamma\in D_\cG(\beta)
    \ssi
 \delta >|\beta-\gamma|_v>0;
\]
\item[(ii)] pour tout choix de $\alpha,\beta,\gamma\in A$ avec
$\beta\in S_\cG(\alpha)$, on a
\[
 \gamma\in D_\cG(\beta)
    \quad\Longleftrightarrow\quad
 |\alpha-\beta|_v >|\beta-\gamma|_v>0.
\]
\end{itemize}
\end{theoreme}

\begin{proof}[D\'emonstration]
Pour chaque $\rho\in R$, on d\'efinit
\[
 A^{(\rho)}=\{\alpha\in A\,;\, |\alpha-\rho|_v<\delta\},
\]
et on choisit un arbre enracin\'e $\cG^{(\rho)} =
(\rho,A^{(\rho)},E^{(\rho)})$ comme dans la proposition
\ref{ultra:prop:arbre}.  Comme les ensembles $A^{(\rho)}$
avec $\rho\in R$ forment une partition de $A$, la r\'eunion
de ces graphes fournit une for\^et enracin\'ee $G=(R,A,E)$
o\`u $E=\cup_{\rho\in R}E^{(\rho)}$.  Par construction, elle
poss\`ede la propri\'et\'e (i). Pour montrer qu'elle
poss\`ede aussi la propri\'et\'e (ii), fixons
$\alpha,\beta,\gamma\in A$ avec $\beta\in S_\cG(\alpha)$.
Soit $\rho\in R$ tel que $\alpha\in A^{(\rho)}$.  Puisque
$\beta\in S_\cG(\alpha)$, on a forc\'ement $\beta\in A^{(\rho)}$
et $D_\cG(\beta)=D_{\cG^{(\rho)}}(\beta)$.  De plus si
$\gamma$ satisfait $|\alpha-\beta|_v >|\beta-\gamma|_v$
alors $|\beta-\gamma|_v<\delta$ et par suite $\gamma\in A^{(\rho)}$.
La condition (ii) sur $\alpha,\beta,\gamma$ est donc satisfaite
dans $\cG$ car elle l'est dans $G^{(\rho)}$.
\end{proof}

On peut reformuler les conditions (i) et (ii) du th\'eor\`eme
en termes de
cha\^{\i}nes \'el\'ementaires de la mani\`ere suivante: pour
$\gamma\in A$, une suite $(\gamma_1,\dots,\gamma_k)$
dans $G$, avec $k\ge 1$, commen\c{c}ant sur une racine
$\gamma_1\in R$, peut
\^etre prolong\'ee en une cha\^{\i}ne \'el\'ementaire
$(\gamma_1,\dots,\gamma_\ell)$ terminant sur
$\gamma_\ell=\gamma$ si et seulement si on a $k=1$ et
$\delta>|\gamma_1-\gamma|_v>0$ ou bien $(\gamma_1,\dots,\gamma_k)$
est une cha\^{\i}ne \'el\'ementaire avec $k\ge 2$  et
$|\gamma_{k-1}-\gamma_k|_v > |\gamma_k-\gamma|_v$.

%
%

\section{Volume des composantes ultram\'etriques}
\label{sec:volultra}

On compl\`ete ici la preuve du th\'eor\`eme \ref{res:principal}
en d\'emontrant les majorations de volume des parties (ii)
et (iii).  Les notations sont celles du paragraphe \ref{sec:resultat}.

\begin{theoreme}
Soit $v$ une place de $K$ au-dessus d'un nombre premier $p$,
soit $\un=(n_1,\dots,n_s)\in\bN_+^s$, et soit $N=n_1+\cdots+n_s$.
Alors le sous-$\cO_v$-module $\cC_{\un,v}$ de $K_v^s$
d\'efini au paragraphe \ref{subsec:res} satisfait
\[
 \mu_v(\cC_{\un,v})^{1/d_v}\le (p^3N)^s |\Delta_\un|_v.
\]
De plus, si $|\alpha_i-\alpha_j|_v=1$ pour tout choix
de $i,j\in\{1,\dots,s\}$ avec $i\neq j$, on a aussi
\[
 \mu_v(\cC_{\un,v})^{1/d_v} = |\Delta_\un|_v.
\]
\end{theoreme}

\begin{proof}[D\'emonstration]
On applique le th\'eor\`eme \ref{ultra:thm:foret} avec
$A=\{\alpha_1,\dots,\alpha_s\}$ et $\delta=p^{-1/(p-1)}$.
Cela fournit une for\^et enracin\'ee $G$ de racines $R$,
de sommets $A$ et d'arcs $E$.  Pour chaque $\alpha\in A$,
on d\'efinit $x_\alpha=x_i$ et $n_\alpha=n_i$ o\`u $i$ est
l'indice pour lequel $\alpha=\alpha_i$.  Alors, $\cC_{\un,v}$ est
contenu dans l'ensemble $\cK_v$ des
points $(x_1,\dots,x_s)\in K_v^s$ satisfaisant
\[
 |x_\beta|_v
    \le p^3 N \prod_{\gamma\in A}
              \max\{|\beta-\gamma|_v,\,\delta\}^{n_\gamma}
\]
pour toute racine $\beta\in R$, ainsi que
\[
 |x_\alpha e^{\beta-\alpha} - x_\beta|_v
   \le p^3 N \prod_{\gamma\in A}
             \max\{|\alpha-\gamma|_v,\,|\beta-\gamma|_v\}^{n_\gamma}
\]
pour tout arc $(\alpha,\beta)\in E$, c'est-\`a-dire pour toute
paire $\alpha,\beta\in A$ avec $\beta\in S_G(\alpha)$ (car
alors on a $|\alpha-\beta|_v<\delta$).  En vertu de la proposition
\ref{graphes:prop}, les $s$ formes lin\'eaires qui d\'efinissent
$\cK_v$ sont lin\'eairement ind\'ependantes, de d\'eterminant $\pm 1$.
Donc $\cK_v$ est un convexe de $K_v^s$ avec
\[
 \mu_v(\cC_{\un,v})^{1/d_v}
   \le \mu_v(\cK_v)^{1/d_v}
   \le (p^3N)^s
      \Delta' \Delta''
\]
o\`u
\[
 \Delta'=\prod_{\substack{\beta\in R\\ \gamma\in A}}
         \max\{|\beta-\gamma|_v,\,\delta\}^{n_\gamma}
 \et
 \Delta''=\prod_{\substack{\alpha,\beta,\gamma\in A\\ \beta\in S_G(\alpha)}}
         \max\{|\alpha-\gamma|_v,\,|\beta-\gamma|_v\}^{n_\gamma}.
\]
Soient $\beta,\gamma\in A$.  Si $\beta\in R$, la condition (i) du
th\'eor\`eme \ref{ultra:thm:foret} livre
\begin{equation}
 \label{volultra:eq:beta-gamma}
 \max\{|\beta-\gamma|_v,\,\delta\}
  =\begin{cases}
     \delta &\text{si $\gamma\in D_G(\beta)\cup\{\beta\}$,}\\
     |\beta-\gamma|_v &\text{sinon.}
   \end{cases}
\end{equation}
Par contre, si $\beta\notin R$, il existe un et un seul $\alpha\in A$
tel que $\beta\in S_G(\alpha)$ et, comme
\[
 |\alpha-\gamma|_v > |\beta-\gamma|_v
 \ssi
 |\alpha-\beta|_v > |\beta-\gamma|_v,
\]
la condition (ii) du m\^eme th\'eor\`eme livre
\begin{equation}
 \label{volultra:eq:alpha-beta-gamma}
 \max\{|\alpha-\gamma|_v,\,|\beta-\gamma|_v\}
  =\begin{cases}
     |\alpha-\gamma|_v &\text{si $\gamma\in D_G(\beta)\cup\{\beta\}$,}\\
     |\beta-\gamma|_v &\text{sinon.}
   \end{cases}
\end{equation}
Comme $D_G(\beta)\cup\{\beta\}$ parcourt les composantes connexes
de $G$ lorsque $\beta$ parcourt $R$ et qu'on a
$\sum_{\gamma\in A} n_\gamma=N$, l'\'egalit\'e
\eqref{volultra:eq:beta-gamma} entra\^{\i}ne
\[
 \Delta'
 = \delta^N
   \prod_{\substack{\beta\in R\\
           \gamma\notin D_G(\beta)\cup\{\beta\}}}
         |\beta-\gamma|_v^{n_\gamma}.
\]
Par ailleurs, l'\'egalit\'e \eqref{volultra:eq:alpha-beta-gamma}
entra\^{\i}ne
\[
 \Delta''
 = \prod_{\substack{\alpha\in A\\ \gamma\in D_G(\alpha)}}
         |\alpha-\gamma|_v^{n_\gamma}
   \prod_{\substack{\beta\notin R\\ \gamma\notin D_G(\beta)\cup\{\beta\}}}
         |\beta-\gamma|_v^{n_\gamma}
\]
On en d\'eduit que
\[
 \Delta'\Delta''
 = \delta^N \prod_{\beta\in A} \prod_{\gamma\in A\setminus\{\beta\}}
    |\beta-\gamma|_v^{n_\gamma}.
\]
Comme $\delta^N=\prod_{\beta\in A} \delta^{n_\beta}
\le \prod_{\beta\in A} |n_\beta!|_v
\le \prod_{\beta\in A} |(n_\beta-1)!|_v$, on conclut que
\[
 \mu_v(\cC_{\un,v})^{1/d_v}
   \le (p^3N)^s
       \prod_{\beta\in A}
         \Big|
         (n_\beta-1)!\prod_{\gamma\neq\beta} (\beta-\gamma)^{n_\gamma}
         \Big|_v
   = (p^3N)^s |\Delta_\un|_v.
\]
Enfin, si $|\alpha_i-\alpha_j|_v=1$ pour tout choix de
$i,j\in\{1,\dots,s\}$ avec $i\neq j$, alors $\cC_{\un,v}$ est l'ensemble
des points $(x_1,\dots,x_s)\in K_v^s$ qui satisfont
\[
 |x_i|_v \le |(n_i-1)!|_v
\]
pour $i=1,\dots,s$, donc
\[
 \mu(\cC_{\un,v})^{1/d_v}=\prod_{i=1}^s |(n_i-1)!|_v = |\Delta_\un|_v.
 \qedhere
\]
\end{proof}

%
%

\section{Un cas particulier}
\label{sec:dexp}

Les convexes ad\'eliques $\cC_\un$ associ\'es \`a un point
$(\alpha_1,\dots,\alpha_s)\in K^s$ d\'ependent seulement des
diff\'erences $\alpha_j-\alpha_i$ avec $1\le i<j\le s$.  Donc,
on peut toujours supposer que $\alpha_1=0$.  Ainsi, pour $s=2$,
on a simplement un couple $(0,\alpha)\in K^2$.  La proposition
ci-dessous pr\'ecise le corollaire \ref{res:cor} pour un tel
point et pour les paires $\un=(n,n)\in\bN_+^2$ de la diagonale.
Dans cet \'enonc\'e le convexe ad\'elique est nomalis\'e de
sorte que sa composante $v$-adique soit contenue dans $\cO_v^2$
pour toute place ultram\'etrique $v$ de $K$.  On l'utilise par
la suite pour d\'emontrer les propositions
\ref{intro:prop:imaginaire} et \ref{intro:prop:e3}.
Les notations sont celles du paragraphe \ref{sec:resultat}.

\begin{proposition}
\label{dexp:prop}
Soit $\alpha\in K\setminus\{0\}$.  On note $E$ l'ensemble fini des places
$v$ de $K$ avec $v\mid\infty$ ou $|\alpha|_v\neq 1$. Pour
chaque place $v$ de $K$ avec $v\nmid\infty$, on pose
$B_v=\min\big\{1,p^{1/(p-1)}|\alpha|_v\big\}$ o\`u $p$ est
le nombre premier tel que $v\mid p$.  On pose encore
\[
 g=\sum_{v\in E} \frac{d_v}{d}
 \et
 B=\prod_{v\nmid \infty} B_v^{-d_v/d}.
\]
Enfin, pour tout entier positif $n\in\bN_+$, on note $\tcC_n$
le convexe ad\'elique de $K^2$ dont les compo\-santes $\tcC_{n,v}$
sont d\'efinies comme suit.
\begin{itemize}
\item[(i)] Si $v\mid\infty$, alors $\tcC_{n,v}$ est l'ensemble des points
$(x,y)\in K_v^2$ tels que
\[
 |x|_v \le n^{g-1}\frac{B^n(2n)!}{|\alpha|_v^n\,n!}
 \et
 |xe^\alpha-y|_v \le n^{g}\frac{B^n|\alpha|_v^n}{4^n\,n!}.
\]
\item[(ii)] Si $v\mid p$ pour un nombre premier $p$ et si
$|\alpha|_v<p^{-1/(p-1)}$, alors $\tcC_{n,v}$ consiste des
points $(x,y)\in K_v^2$ tels que
\[
 |x|_v \le 1
 \et
 |xe^\alpha-y|_v \le B_v^{2n}.
\]
\item[(iii)] Si $v\mid p$ pour un nombre premier $p$ et si
$|\alpha|_v\ge p^{-1/(p-1)}$, alors $\tcC_{n,v}=\cO_v^2$.
\end{itemize}
Alors on a
\begin{equation}
\label{dexp:prop:eq1}
 c_4 n^{-2g+1} \le \lambda_1(\tcC_n)\le \lambda_2(\tcC_n) \le c_3
\end{equation}
pour des constantes $c_3,c_4>0$ ind\'ependantes de $n$.
\end{proposition}

\begin{proof}[D\'emonstration]
Soit $n\in\bN_+$. On consid\`ere le convexe ad\'elique
$\cC_\un$ construit au paragraphe \ref{subsec:res} pour le
choix de $\alpha_1=0$, $\alpha_2=\alpha$ et $\un=(n,n)$.  Pour une
place archim\'edienne $v$ de $K$ associ\'ee \`a un plongement
$\sigma\colon K\hookrightarrow\bC$ et pour $k=1,2$, on trouve
\begin{align*}
 \left|\int_0^{\sigma(\alpha)}
       f_{\un-\ue_k}^\sigma(z)e^{\sigma(\alpha)-z} dz \right|
 &\le |\sigma(\alpha)| e^{|\sigma(\alpha)|}
      \max_{t\in[0,1]}
        \left|f_{\un-\ue_k}^\sigma(\sigma(\alpha)t)\right|\\
 &\le e^{|\sigma(\alpha)|} |\sigma(\alpha)|^{2n}
      \max_{t\in[0,1]} t^{n-1}(1-t)^{n-1}\\
 &= 4 e^{|\alpha|_v} (|\alpha|_v/2)^{2n}.
\end{align*}
Donc les points $(x,y)$ de $\cC_{\un,v}$ satisfont
\[
 |x|_v\le e^{|\alpha|_v}(2n-1)!
 \et
 |xe^\alpha-y|_v\le 4e^{|\alpha|_v}(|\alpha|_v/2)^{2n}.
\]
On en d\'eduit que $a_v\cC_{\un,v} \subseteq \tcC_{n,v}$
en posant
\[
 a_v=\frac{n^{g-1}B^n}{4e^{|\alpha|_v} \alpha^n (n-1)!}
    \in K_v^\times.
\]
Pour tout nombre premier $p$ et toute place $v$ de $K$
avec $v\mid p$, on trouve aussi que
$a_v\cC_{\un,v} \subseteq \tcC_{n,v}$ en posant
\[
 a_v=\frac{p^{t_v}}{\alpha^n(n-1)!} \in K_v^\times
\]
o\`u $t_v$ est l'entier pour lequel
\[
 2np^3B_v^{-n}\le p^{t_v} < 2np^4B_v^{-n}
\]
si $v\in E$ et o\`u $t_v=0$ sinon.  Ce calcul utilise
simplement le fait que
$|(n-1)!|_v\ge |n!|_v\ge p^{-n/(p-1)}$.  Ainsi on
obtient $a\,\cC_\un\subseteq\tcC_n$ pour l'id\`ele
$a=(a_v)_v \in K_\bA^\times$.

Le produit $\cD=\prod_v \{x\in K_v\,;\, |x|_v\le |a_v|_v\}
\subset K_\bA$
est un convexe ad\'elique de $K$.  En appliquant la formule
du produit \`a l'id\`ele principale $\alpha^n(n-1)!\in K^\times$,
on trouve que le volume de $\cD$ est
\[
 \mu(\cD) = 2^{r_1}\pi^{r_2}\prod_v |a_v|_v^{d_v}
   = 2^{r_1}\pi^{r_2}
     \prod_{v\mid\infty}
       \Big(\frac{n^{g-1}B^n}{4e^{|\alpha|_v}}\Big)^{d_v}
     \prod_{p,\,v\mid p} p^{-t_vd_v}.
\]
Comme $\prod_{v\mid\infty} B^{d_v} = B^d =
\prod_{v\nmid\infty}B_v^{-d_v}$, cette formule se r\'e\'ecrit
\[
 \mu(\cD) = c_1 n^{d(g-1)} \prod_{p,\,v\mid p} (p^{t_v}B_v^n)^{-d_v},
\]
avec $c_1=2^{r_1}\pi^{r_2} \prod_{v\mid\infty} (4e^{|\alpha|_v})^{-d_v}$.
Comme $p^{t_v}B_v^n=1$ si $v\notin E$ et que $p^{t_v}B_v^n<2np^4$ si
$v\in E$ avec $v\mid p$, on en tire
\[
 \mu(\cD)
   \ge c_2 n^{d(g-1)} \prod_{v\in E'} n^{-d_v}
   = c_2 n^{dg} \prod_{v\in E} n^{-d_v}
   = c_2,
\]
o\`u $E'=\{v\in E\,;\,v\nmid\infty\}$ et
$c_2=c_1\prod_{v\in E',\, v\mid p}(2p^4)^{-d_v}$.
En vertu du th\'eor\`eme \ref{res:thm:MBV} (avec $s=1$), on a donc
$\lambda_1(\cD)\le c_3$ o\`u $c_3=(2^{r_1+r_2}|D(K)|^{1/2}c_2^{-1})^{1/d}$.
Autrement dit, il existe $\beta\in K^\times$ tel que, pour toute
place $v$ de $K$ on ait $|\beta|_v\le c_3|a_v|_v$ si $v\mid\infty$
et $|\beta|_v\le |a_v|_v$ sinon.  Par suite, on obtient
\[
 \beta\cC_\un\subseteq c_3\tcC_n,
\]
ce qui entra\^{\i}ne aussit\^ot
\[
 \lambda_1(\tcC_n)\le \lambda_2(\tcC_n)\le c_3
\]
car $\beta\cC_\un$ contient les points $\beta\ua_{\un-\ue_1},\beta\ua_{\un-\ue_2}
\in K^2$ qui sont lin\'eairement ind\'ependants sur $K$.
En vertu du th\'eor\`eme \ref{res:thm:MBV} (avec $s=2$), cela implique
\[
 \lambda_1(\tcC_n)
   \ge (2c_3)^{-1}\mu(c_3\tcC_n)^{-1/d}
     = (2c_3^2)^{-1}\mu(\tcC_n)^{-1/d}.
\]
Enfin, pour toute place $v$ de $K$, on trouve
\[
 \mu_v(\tcC_{n,v})^{1/d_v}
   \le \begin{cases}
         4n^{2g-1}B^{2n}
         &\text{si $v\mid\infty$,}\\
         B_v^{2n}
         &\text{sinon.}
        \end{cases}
\]
Comme $B^d\prod_{v\nmid\infty}B_v^{d_v}=1$, cela implique que
$\mu(\tcC_n)^{1/d}\le 4n^{2g-1}$, et on obtient
\eqref{dexp:prop:eq1} avec $c_4=(8c_3^2)^{-1}$.
\end{proof}

\begin{proof}[D\'emonstration de la proposition
\ref{intro:prop:imaginaire}]
Sous les hypoth\`eses de cette proposition, le corps
$K$ poss\`ede une seule place archim\'edienne $\infty$,
induite par l'inclusion
$K\subset\bC$.  De plus, dans les notations de la
proposition \ref{dexp:prop}, le choix de $\alpha$
entra\^{\i}ne $B_v=1$ pour toute autre place $v$ de
$K$, donc, pour tout $n\in\bN_+$, on a
\[
 \tcC_n=\tcC_{n,\infty}\times\prod_{v\neq\infty}\cO_v^2,
\]
o\`u $\tcC_{n,\infty}$ est l'ensemble des points $(x,y)$
de $K_\infty^2\subseteq\bC^2$ qui satisfont
\[
 |x| \le n^{g-1}\frac{(2n)!}{|\alpha|^n\,n!}
 \et
 |xe^\alpha-y| \le n^{g}\frac{|\alpha|^n}{4^n\,n!}.
\]
De plus, en vertu de \eqref{dexp:prop:eq1}, on a
$\lambda_1(\tcC_n)\ge c_4n^{-2g+1}$ pour une constante
$c_4>0$ ind\'ependante de $n$.

Soit $(x,y)\in\cO_K^2$ avec $x\neq 0$.    Ce qui pr\'ec\`ede
implique que, pour tout $n\in\bN_+$,
\[
 \text{si} \quad |x| < h(n) := c_4 n^{-g}\frac{(2n)!}{|\alpha|^n\,n!}
 \quad\text{alors}\quad
 |xe^\alpha-y| \ge c_4 n^{-g+1}\frac{|\alpha|^n}{4^n\,n!} \,.
\]
Si $|x|$ est assez grand, on peut \'ecrire $h(n-1)\le |x|<h(n)$ pour un
entier $n\ge 2$ avec $h(n-1)\ge e^n$.  Alors on a $n\le \log |x|$ et on obtient
\begin{align*}
 |x|\,|xe^\alpha-y|
   &\ge h(n-1)c_4 n^{-g+1}\frac{|\alpha|^n}{4^n\,n!} \\
    &\ge c_4^2|\alpha| n^{-2g}\binom{2n-2}{n-1}4^{-n}
      \ge c_5 n^{-2g-1}
      \ge c_5 (\log |x|)^{-2g-1},
\end{align*}
avec $c_5=c_4^2|\alpha|/8$.  Comme $\cO_K$ est un sous-ensemble discret
de $\bC$, cela laisse de c\^ot\'e un nombre fini de valeurs de $x$ et
pour les inclure dans la minoration finale il suffit
de remplacer $c_5$ par une constante $c>0$ suffisamment petite.
\end{proof}

\begin{proof}[D\'emonstration de la proposition
\ref{intro:prop:e3}]
On applique la proposition \ref{dexp:prop} avec $K=\bQ$ et
$\alpha=3$.  Dans ce cas, on a $g=2$ et $B=B_3^{-1}=3^{1/2}$.
Alors, pour $n\in\bN_+$, un calcul simple montre que la
composante archim\'edienne $\tcC_{n,\infty}$ du convexe
ad\'elique $\tcC_n$ satisfait
\begin{equation}
 \label{dexp:e3:eq1}
 n\cC_n \subseteq \tcC_{n,\infty}\subseteq n^2\cC_n,
\end{equation}
o\`u $\cC_n$ est le corps convexe de $\bR^2$ d\'efini dans la proposition
\ref{intro:prop:e3}.  Pour les composantes ultra\-m\'e\-triques, on trouve
\[
 \tcC_{n,3}=\{(x,y)\in\bZ_3\,;\, |xe^3-y|_3\le 3^{-n}\}
\]
et $\tcC_{n,p}=\bZ_p^2$ pour tout nombre premier $p$ avec
$p\neq 3$.   Donc les points de $\bQ^2$ qui appartiennent \`a
chacune de ces composantes sont pr\'ecis\'ement les points
du r\'eseau $\Lambda_n$ de la proposition \ref{intro:prop:e3}.  Par suite, les
minima de $\tcC_n$ par rapport \`a $\bQ^2$ au sens ad\'elique sont
aussi les minima de $\tcC_{n,\infty}$ par rapport \`a $\Lambda_n$
au sens classique.  En vertu des inclusions \eqref{dexp:e3:eq1}, cela
implique que $c_4n^{-2}\le \lambda_1(\cC_n,\Lambda)\le
\lambda_2(\cC_n,\Lambda)\le c_3n^2$ pour les constantes $c_3$ et $c_4$
donn\'ees par la proposi\-tion \ref{dexp:prop}.
\end{proof}

%
%

\section{Calculs num\'eriques}
\label{sec:num}

Les formules de
l'appendice \ref{sec:rel} permettent de calculer r\'ecursi\-ve\-ment
les approximations $\ua_{n-1,n}$ et $\ua_{n,n-1}$
de $(1,e^3)$ pour chaque entier $n\ge 1$.  Dans ce paragraphe,
on montre comment ces derni\`eres permettent \`a leur tour de calculer
efficacement les quotients partiels du d\'eveloppement en
fraction continue de $e^3\in\bR$, puis de v\'erifier
les in\'egalit\'es \eqref{intro:eq:loglog} de l'introduction.
Notre r\'ef\'erence pour les fractions continues est
\cite[Ch.~I]{Sc1980}.

Soit $e^3=[a_0,a_1,a_2,\dots]$ le d\'eveloppement en fraction
continue de $e^3$.  Ses premiers termes sont
\[
 e^3=[20, 11, 1, 2, 4, 3, 1, 5, 1, 2, 16,\dots ],
\]
sans r\'egularit\'e apparente.  Pour chaque entier $n\ge 0$,
on forme la $n$-i\`eme r\'eduite de $e^3$
\[
 \frac{p_n}{q_n}=[a_0,a_1,\dots,a_n]
\]
avec $p_n\in\bZ$, $q_n\in\bN$ et $\gcd(p_n,q_n)=1$.
Le tableau ci-dessous liste les entiers $n\ge 1$ avec
$q_{n-1}\le 10^{500\,000}$ pour lesquels
\[
 a_n=\max\{a_1,\dots,a_n\}.
\]
Pour chacun de ces entiers, il donne la valeur correspondante
de $a_n$ et une approximation de $\log(q_{n-1})$ tronqu\'ee
\`a la premi\`ere d\'ecimale.

\[
\begin{array}{r|ccccccccccc}
 n&1&10&31&87&133&211&244&388&2708&8055\\[2pt]
 \hline\\[-10pt]
 a_n&11&16&68&189&492&739&2566&5885&6384&10409\\[2pt]
 \log(q_{n-1})& 0.0& 9.4& 34.5& 97.9& 151.1& 256.6& 297.6&
 475.0& 3307.2& 9614.8
 \end{array}
\]

\[
\begin{array}{r|ccccccc}
 n&9437&29508&30939&43482&91737&196440&476544\\[2pt]
 \hline\\[-10pt]
 a_n&19362&21981&46602&51140&315466&546341&569869\\[2pt]
 \log(q_{n-1})&11258.4& 34996.8& 36750.6& 51515.4& 109063.1& 233261.9& 566111.1
\end{array}
\]

Il est facile d'en d\'eduire les estimations
\eqref{intro:eq:loglog}.  En effet, posons $\psi(x)=3\log(x)\log(\log(x))$
pour tout $x\ge e$.  Pour chaque couple $(p,q)\in\bZ^2$ avec $q\ge 1$,
il existe un entier $n\ge 1$ tel que $q_{n-1} \le q < q_n$.  En vertu
d'un th\'eor\`eme de Lagrange \cite[Ch.~I, Thm.~5E]{Sc1980}, on a
\[
 |qe^3-p|
   \ge |q_{n-1}e^3-p_{n-1}|
   \ge \frac{1}{q_n+q_{n-1}}
   \ge \frac{1}{(a_n+2)q_{n-1}}.
\]
En supposant $q\ge 3$, on en d\'eduit
\begin{equation}
 \label{num:eq1}
 \psi(q)q\,|qe^3-p|
   \ge \frac{\psi(q_{n-1})}{a_n+2}.
\end{equation}
On v\'erifie sans peine que le membre de droite de
\eqref{num:eq1} est $\ge 1$ pour tous les entiers
$n$ de la table avec $n\ge 10$, donc il l'est aussi
pour tout entier $n\ge 10$ avec $q_{n-1}\le 10^{500\,000}$.
Un calcul rapide montre que c'est
encore vrai pour $n=2,\dots,9$.  Donc le membre de
gauche de \eqref{num:eq1} est $\ge 1$ si $11\le q\le 10^{500\,000}$.
Enfin, on v\'erifie que c'est encore le cas si
$4\le q\le 10$.

Pour calculer les quotients partiels $a_n$, on pose
\[
 C_n=\begin{pmatrix} 2n-4 &2n-1\\ 2n-1 &2n+2 \end{pmatrix}
 \et
 A_n=C_n\cdots C_1
\]
pour chaque entier $n\ge 1$.  Le corollaire \ref{rel:cor2}
de l'appendice montre que les lignes de $(n-1)!A_n$ sont
les approximations d'Hermite $\ua_{n-1,n}$ et $\ua_{n,n-1}$
de $(1,e^3)$ pour tout $n\in\bN_+$, donc on a
\begin{equation}
 \label{num:eq2}
 \lim_{n\to\infty} A_n\begin{pmatrix} e^3\\-1\end{pmatrix}
 = \begin{pmatrix} 0\\0\end{pmatrix}\,.
\end{equation}
On note aussi que, pour tout $n\ge 2$, les matrices
$C_n$ et $A_n$ appartiennent \`a l'ensemble
\[
 \cM
  =\left\{
    \begin{pmatrix}t&u\\t'&u'\end{pmatrix}\in\Mat_{2\times 2}(\bZ)
    \,;\, 0\le t<u,\ 0\le t'<u' \et tu'\neq t'u
   \right\}.
\]
C'est clair pour les $C_n$.  Quant aux $A_n$, cela d\'ecoule du
fait que l'ensemble $\cM$ est ferm\'e sous le produit.

En g\'en\'eral, si $A=\begin{pmatrix}t&u\\t'&u'\end{pmatrix}
\in\cM$, les nombres rationnels $t/u$ et $t'/u'$ admettent
des d\'eveloppements en fractions continues
\[
 \frac{t}{u}=[a_0,a_1,\dots,a_\ell]
 \et
 \frac{t'}{u'}=[a'_0,a'_1,\dots,a'_{\ell'}]
\]
avec $a_0=a'_0=0$ et $a_\ell,a'_{\ell'}\ge 2$.  Soit
$(a_0,\dots,a_k)$ la partie initiale commune des suites
$(a_0,\dots,a_\ell)$ et $(a'_0,\dots,a'_{\ell'})$.  Si $k=0$,
ce qui a lieu si $t=0$ ou $t'=0$ ou
$\lfloor u/t\rfloor\neq \lfloor u'/t'\rfloor$,
on dit que $A$ est \emph{r\'eduite}. Sinon, on v\'erifie
sans peine que
\[
 A = R\begin{pmatrix}0&1\\1&a_k\end{pmatrix}
     \cdots\begin{pmatrix}0&1\\1&a_1\end{pmatrix}
\]
o\`u $R\in\cM$ est r\'eduite. Si $k=0$, cette formule tient
encore en convenant que le produit de droite est $R$.
En particulier, pour chaque $n\ge 2$, on peut \'ecrire
\[
 A_n=R_n\begin{pmatrix}0&1\\1&a_{k(n)}\end{pmatrix}
     \cdots\begin{pmatrix}0&1\\1&a_1\end{pmatrix}
\]
pour une matrice r\'eduite $R_n\in\cM$, des entiers
$0\le k(1)\le k(2)\le \cdots$ et des entiers positifs
$a_1,a_2,\dots$ tels que
\begin{equation}
 \label{rel:eq:recurrence}
 C_{n+1}R_n=R_{n+1}\begin{pmatrix}0&1\\1&a_{k(n+1)}\end{pmatrix}
     \cdots\begin{pmatrix}0&1\\1&a_{k(n)+1}\end{pmatrix},
\end{equation}
avec la convention que le produit de droite est $R_{n+1}$
si $k(n+1)=k(n)$.  Gr\^ace \`a \eqref{num:eq2}, la suite
des $k(n)$ est non born\'ee et on conclut que
\[
 e^{-3}=[0,a_1,a_2,\dots]
 \et
 e^3=[a_1,a_2,\dots]
\]
sont les d\'eveloppements en fractions continues respectifs
de $e^{-3}$ et de $e^3$.  Pour le calcul des $a_n$, il suffit
donc de construire r\'ecursivement les matrices $R_n$ dont
la norme est en pratique beaucoup plus petite que celle de $A_n$
(on peut m\^eme \`a chaque pas extraire les puissances de $3$
qui divisent $R_n$).  Pour les entiers $q_n$, on sauve aussi
du temps de calcul en ne conservant que leur repr\'esentation
en point flottant (nous avons utilis\'e 10 chiffres significatifs).
Il faut alors un tout petit plus d'une
heure de temps CPU pour produire les tables sous le logiciel
MAPLE avec un processeur intel i5 de 64 bits.

%
%

\appendix

\section{Relations de r\'ecurrence}
\label{sec:rel}

On reprend les notations du paragraphe \ref{subsec:approx:Hermite}
et on \'etend la d\'efinition de $f_\un(z)$, $P_\un(z)$ et
$a_\un$ \`a tout $s$-uplet $\un\in\bZ^s$ en posant
\[
 f_\un(z)=P_\un(z)=0
 \et
 a_\un=(0,\dots,0)
 \quad\text{si}\quad
 \un\notin\bN^s.
\]
Pour tout $\un\in\bN_+^s$, on note $A_\un$ la matrice dont la
$\ell$-i\`eme ligne est $\ua_{\un-\ue_\ell}$ pour $\ell=1,\dots,s$.
Dans \cite[\S\S IX-X]{He1873}, Hermite \'etablit une formule de 
r\'ecurrence liant $A_{\un+\uun}$ et $A_\un$ o\`u $\uun$ d\'esigne
le $s$-uplet $(1,\dots,1)$. On donne ici des formules 
de r\'ecurrence plus g\'en\'erales bas\'ees sur le m\^eme principe.  
La relation \eqref{rel:eq1} ci-dessous est d\^ue \`a Hermite 
\cite[\S IX, p.~230]{He1873} dans le cas o\`u $\un\in\bN_+^s$.

\begin{proposition}
\label{rel:prop}
Soit $\un=(n_1,\dots,n_s)\in\bN^s$. On a
\begin{equation}
 \label{rel:eq1}
 a_\un = (f_\un(\alpha_1),\dots,f_\un(\alpha_s))
          + \sum_{j=1}^s n_ja_{\un-\ue_j}\,.\\
\end{equation}
De plus, si $k,\ell\in\{1,\dots,s\}$ avec $n_k\ge 1$,
on a aussi
\begin{equation}
 \label{rel:eq2}
 a_{\un+\ue_\ell-\ue_k}
   = a_\un + (\alpha_k-\alpha_\ell)a_{\un-\ue_k}.
\end{equation}
\end{proposition}

\begin{proof}[D\'emonstration]
La formule pour la d\'eriv\'ee d'un produit donne
\[
 f_\un'(z)=\sum_{j=1}^s n_j f_{\un-\ue_j}(z).
\]
En prenant la somme des d\'eriv\'ees des deux membres
de cette \'egalit\'e, on obtient
\[
 P_\un(z)=f_\un(z)+\sum_{j=1}^s n_j P_{\un-\ue_j}(z)
\]
et \eqref{rel:eq1} s'ensuit.  La formule \eqref{rel:eq2}
est imm\'ediate si $k=\ell$.
Supposons $k\neq \ell$ et $n_k\ge 1$ de sorte que
$\un-\ue_k\in\bN^s$. Alors on a
\[
 f_{\un+\ue_\ell-\ue_k}(z) - f_\un(z)
   = (z-\alpha_\ell)f_{\un-\ue_k}(z)-(z-\alpha_k)f_{\un-\ue_k}(z)
   = (\alpha_k-\alpha_\ell)f_{\un-\ue_k}(z).
\]
En prenant de nouveau la somme des d\'eriv\'ees, on obtient
\[
 P_{\un+\ue_\ell-\ue_k}(z)
   = P_\un(z) + (\alpha_k-\alpha_\ell)P_{\un-\ue_k}(z)
\]
et \eqref{rel:eq2} s'ensuit.
\end{proof}

\begin{cor}
\label{rel:cor1}
Soient $\un=(n_1,\dots,n_s)\in \bN_+^s$ et $\ell\in\{1,\dots,s\}$.
Alors, on a
\[
 A_{\un+\ue_\ell} = M_{\un,\ell} A_\un
\]
o\`u
\[
 M_{\un,\ell}
   =\begin{pmatrix}
     n_1+(\alpha_1-\alpha_\ell) &n_2 &\cdots &n_s\\
     n_1 &n_2+(\alpha_2-\alpha_\ell) &\cdots &n_s\\
     \vdots &\vdots&\ddots&\vdots\\
     n_1 &n_2 &\cdots &n_s+(\alpha_s-\alpha_\ell)
    \end{pmatrix}.
\]
\end{cor}

\begin{proof}[D\'emonstration]
Comme les coordonn\'ees de $\un$ sont
positives, le polyn\^ome $f_\un$ s'annule aux
points $\alpha_1,\dots,\alpha_s$ et les formules de
la proposition \ref{rel:prop} livrent
\[
 a_{\un+\ue_\ell-\ue_k}
   = (\alpha_k-\alpha_\ell)a_{\un-\ue_k} + \sum_{j=1}^s n_j a_{\un-\ue_j}
 \quad
 (1\le k\le s).
\qedhere
\]
\end{proof}

Dans le cas o\`u $s=2$, on en d\'eduit une formule tr\`es simple
pour les approximations diagonales $A_{n,n}$.

\begin{cor}
\label{rel:cor2}
Supposons que $s=2$, $\alpha_1=0$ et $\alpha_2=\alpha\in K\setminus\{0\}$.
Pour tout $n\in\bN_+$, on a
\begin{equation}
 \label{rel:cor2:eq}
 A_{n,n}
  = \begin{pmatrix}
      P_{n-1,n}(0) &P_{n-1,n}(\alpha)\\
      P_{n,n-1}(0) &P_{n,n-1}(\alpha)
    \end{pmatrix}
  = (n-1)! C_nC_{n-1}\cdots C_1
\end{equation}
o\`u
\[
 C_i
   =\begin{pmatrix}
     2i-1-\alpha &2i-1\\
     2i-1 &2i-1+\alpha
    \end{pmatrix}
 \quad
 (i\in\bN_+).
\]
\end{cor}

\begin{proof}[D\'emonstration]
On trouve que $P_{0,1}(z)=z+1-\alpha$ et $P_{1,0}(z)=z+1$,
donc $A_{1,1}=C_1$.  En g\'en\'eral, pour un entier $n\ge 1$,
les formules du corollaire \ref{rel:cor1} livrent
\[
 A_{n+1,n+1}
  = \begin{pmatrix}
     n &n+1\\
     n &n+1+\alpha
    \end{pmatrix}
    \begin{pmatrix}
     n-\alpha &n\\
     n &n
    \end{pmatrix}
    A_{n,n}
  = n C_{n+1} A_{n,n},
\]
et la conclusion suit par r\'ecurrence sur $n$.
\end{proof}

\vfill
 \vfill

\small
\vbox{
\hbox{Damien \sc Roy}\par
\hbox{D\'epartement de math\'ematiques et de statistique}\par
\hbox{Universit\'e d'Ottawa}\par
\hbox{150 Louis Pasteur}\par
\hbox{Ottawa, Ontario}\par
\hbox{Canada K1N 6N5}
}

\end{document}